\documentclass[reqno]{amsart}

\usepackage{amssymb}
\usepackage{enumitem}
\usepackage[normalem]{ulem}

\newcommand\numberthis{\addtocounter{equation}{1}\tag{\theequation}}

\usepackage{tikz}
\usetikzlibrary{decorations.pathmorphing, arrows, calc}
\usetikzlibrary{shapes,arrows}

\usepackage{enumitem} 

\usepackage{subcaption}

\newcommand{\script}{\mathcal}
\newcommand{\parentheses}[1]{{\left( {#1} \right)}}

\newcommand{\p}{\parentheses}

\newcommand{\cl}{\mathrm{cl}}
\newcommand{\Set}[1]{{\left\lbrace {#1} \right\rbrace}}
\newcommand{\singleton}{\Set}

\newcommand{\cardinality}[1]{{\left\lvert {#1} \right\rvert}}

\def\set#1:#2{\Set{{#1} \colon {#2}}}

\newcommand{\rooot}[1]{\text{r}\p{#1}}

\newtheorem{theorem}{Theorem}\numberwithin{theorem}{section} 

\newtheorem{lemma}[theorem]{Lemma}

\newtheorem{prop}[theorem]{Proposition}

\newtheorem{defn}[theorem]{Definition}

\newtheorem{prob}[theorem]{Problem}
\newtheorem{conj}[theorem]{Conjecture}

\newtheorem{claim}[theorem]{Claim}

\newcommand{\N}{\mathbb{N}}

\DeclareMathOperator{\Ball}{Ball}

\tikzset{
square/.pic={
\node[draw,circle,scale=.3,fill] (b1) at (0,-2){};
\node[draw,circle,scale=.3,fill] (b2) at (0,-3){};
\node[draw,circle,scale=.3,fill] (start) at (0,-5){};
\node[draw,green,circle,scale=.3,fill] (green) at (-1,-2){};
\node[draw,yellow,circle,scale=.3,fill] (yellow) at (-1,-3){};
\draw[dotted] (0,0)--(b1);
\draw[dotted] (b2)--(start);
\draw (green)--(b1)--(b2)--(yellow);
\draw (start)--(0,-6);
\draw (-1,-6)--(1,-6)--(1,-7.6)--(1.3,-7.4)--(1.3,-8)--(-0.7,-8)--(-0.7,-6.4)--(-1,-6.6)--(-1,-6)--(1,-6);
}
}
\tikzset{
triangle/.pic={
\node[draw,circle,scale=.3,fill] (b1) at (0,-2){};
\node[draw,circle,scale=.3,fill] (b2) at (0,-3){};
\node[draw,circle,scale=.3,fill] (start) at (0,-5){};
\node[draw,yellow,circle,scale=.3,fill] (green) at (-1,-2){};
\node[draw,green,circle,scale=.3,fill] (yellow) at (-1,-3){};
\draw[dotted] (0,0)--(b1);
\draw[dotted] (b2)--(start);
\draw (green)--(b1)--(b2)--(yellow);
\draw (start)--(0,-6);
\draw (0,-6)--(1,-8)--(-1,-8)--(0,-6);
}
}
\tikzset{
decsquare/.pic={
\pic at (0,0) {square};
\pic[scale = 0.3,every node/.style={transform shape}] at (-0.6,-8) {square};
\pic[scale = 0.3, every node/.style={transform shape}] at (0.3,-8) {triangle};
\pic[scale = 0.3, every node/.style={transform shape}] at (1.2,-8) {square};
}
}
\tikzset{
dectriangle/.pic={
\pic at (0,0) {triangle};
\pic[scale = 0.3, every node/.style={transform shape}] at (-0.9,-8) {triangle};
\pic[scale = 0.3, every node/.style={transform shape}] at (0.0,-8) {triangle};
\pic[scale = 0.3, every node/.style={transform shape}] at (0.9,-8) {square};
}
}
\tikzset{
decdecsquare/.pic={
\pic at (0,0) {square};
\pic[scale = 0.3, every node/.style={transform shape}] at (-0.6,-8) {decsquare};
\pic[scale = 0.3, every node/.style={transform shape}] at (0.3,-8) {dectriangle};
\pic[scale = 0.3, every node/.style={transform shape}] at (1.2,-8) {decsquare};
}
}
\tikzset{
decdectriangle/.pic={
\pic at (0,0) {triangle};
\pic[scale = 0.3, every node/.style={transform shape}] at (-0.9,-8) {dectriangle};
\pic[scale = 0.3, every node/.style={transform shape}] at (0,-8) {dectriangle};
\pic[scale = 0.3, every node/.style={transform shape}] at (0.9,-8) {decsquare};
}
}

\title[Non-reconstructible locally finite trees]{A counterexample to the reconstruction conjecture for locally finite trees}

\author[Bowler, Erde, Heinig, Lehner \& Pitz]{Nathan Bowler, Joshua Erde, Peter Heinig, Florian Lehner, Max Pitz}
\address{University of Hamburg, Department of Mathematics, Bundesstra{\ss}e 55 (Geomatikum), 20146 Hamburg, Germany}
\email{nathan.bowler@uni-hamburg.de, joshua.erde@uni-hamburg.de, \newline  peter.heinig@math.uni-hamburg.de, florian.lehner@uni-hamburg.de, max.pitz@uni-hamburg.de}



\begin{document}

\begin{abstract}
Two graphs $G$ and $H$ are \emph{hypomorphic} if there exists a bijection $\varphi \colon V(G) \rightarrow V(H)$ such that $G - v \cong H - \varphi(v)$ for each $v \in V(G)$. A graph $G$ is \emph{reconstructible} if $H \cong G$ for all $H$ hypomorphic to $G$. 

It is well known that not all infinite graphs are reconstructible. However, the Harary-Schwenk-Scott Conjecture from 1972 suggests that all locally finite trees are reconstructible. 

In this paper, we construct a counterexample to the Harary-Schwenk-Scott Conjecture. Our example also answers four other questions of Nash-Williams, Halin and Andreae on the reconstruction of infinite graphs.
\end{abstract}
\maketitle

\section{Introduction}
We say that two graphs $G$ and $H$ are \emph{(vertex-)}hypomorphic if there exists a bijection $\varphi$ between the vertices of $G$ and $H$ such that the induced subgraphs $G - v$ and $H - \varphi(v)$ are isomorphic for each vertex $v$ of $G$. Any such bijection is called a \emph{hypomorphism}. We say that a graph $G$ is \emph{reconstructible} if $H \cong G$ for every $H$ hypomorphic to $G$. The following conjecture, attributed to Kelly and Ulam, is perhaps one of the most famous unsolved problems in the theory of graphs.

\begin{conj}[The Reconstruction Conjecture]
Every finite graph with at least three vertices is reconstructible.
\end{conj}

For an overview of results towards the Reconstruction Conjecture for finite graphs see the survey of Bondy and Hemminger \cite{BH77}. Harary \cite{H64} proposed the Reconstruction Conjecture for infinite graphs, however Fisher \cite{F69} found a counterexample, which was simplified to the following counterexample by Fisher, Graham and Harary \cite{FGH72}: consider the infinite tree $G$ in which every vertex has countably infinite degree, and the graph $H$ formed by taking two disjoint copies of $G$, which we will write as $G \sqcup G$. For each vertex $v$ of $G$, the induced subgraph $G - v$ is isomorphic to $G \sqcup G \sqcup \cdots$, a disjoint union of countably many copies of $G$, and similarly for each vertex $w$ of $H$, the induced subgraph $H - w$ is isomorphic to $G \sqcup G \sqcup \cdots$ as well. Therefore, any bijection from $V(G)$ to $V(H)$ is a hypomorphism, but $G$ and $H$ are clearly not isomorphic. Hence, the tree $G$ is not reconstructible.

These examples, however, contain vertices of infinite degree. Regarding locally finite graphs, Harary, Schwenk and Scott \cite{HSS72} showed that there exists a non-reconstructible locally finite forest. However, they conjectured that the Reconstruction Conjecture should hold for locally finite trees.

\begin{conj}[The Harary-Schwenk-Scott Conjecture]
\label{hssconj}
Every locally finite tree is reconstructible.
\end{conj}

This conjecture has been verified in a number of special cases. Kelly \cite{kelly1957} showed that finite trees on at least three vertices are reconstructible. Bondy and Hemminger \cite{BH74} showed that every tree with at least two but a finite number of ends is reconstructible, and Thomassen \cite{T78} showed that this also holds for one-ended trees. Andreae \cite{A81} proved that also every tree with countably many ends is reconstructible.

A survey of Nash-Williams \cite{NW91} on the subject of reconstruction problems in infinite graphs gave the following three main open problems in this area, which have remained open until now.

\begin{prob}[Nash-Williams]\label{p:one}
Is every locally finite connected infinite graph reconstructible?
\end{prob}

\begin{prob}[Nash-Williams]\label{p:two}
If two infinite trees are hypomorphic, are they also isomorphic?
\end{prob}

\begin{prob}[Halin]\label{p:three}
If $G$ and $H$ are hypomorphic, do there exist embeddings $G \hookrightarrow H$ and $H \hookrightarrow G$?
\end{prob}

Problem~\ref{p:two} has been emphasized in Andreae's \cite{ANDREAE1994183}, which contains partial affirmative results on Problem~\ref{p:two}. A positive answer to Problem~\ref{p:one} or \ref{p:two} would verify the Harary-Schwenk-Scott Conjecture. In this paper we construct a pair of trees which are not only a counterexample to the Harary-Schwenk-Scott Conjecture, but also answer the three questions of Nash-Williams and Halin in the negative. Our counterexample will in fact have bounded degree.

\begin{theorem}\label{t:main}
There are two (vertex)-hypomorphic infinite trees $T$ and $S$ with maximum degree three such that there is no embedding $T \hookrightarrow S$ or $S \hookrightarrow T$.
\end{theorem}

Our example also provides a strong answer to a question by Andreae \cite{ANDREAE1982258} about edge-reconstructibility. Two graphs $G$ and $H$ are \emph{edge-hypomorphic} if there exists a bijection $\varphi \colon E(G) \rightarrow E(H)$ such that $G - e \cong H - \varphi(e)$ for each $e \in E(G)$. A graph $G$ is \emph{edge-reconstructible} if $H \cong G$ for all $H$ edge-hypomorphic to $G$. In \cite{ANDREAE1982258} Andreae constructed countable forests which are not edge-reconstructible, but conjectured that no locally finite such examples can exist.

\begin{prob}[Andreae]\label{p:four1}
Is every locally finite graph with infinitely many edges edge-reconstructible?
\end{prob}

Our example answers Problem~\ref{p:four1} in the negative: the trees $T$ and $S$ we construct for Theorem~\ref{t:main} will also be edge-hypomorphic. Besides answering Problem~\ref{p:four1}, this appears to be the first known example of two non-isomorphic graphs that are \emph{simultaneously} vertex- and edge-hypomorphic.

The Reconstruction Conjecture has also been considered for general locally finite graphs. Nash-Williams \cite{NW87} showed that any locally finite graph with at least three, but a finite number of ends is reconstructible, and in \cite{NW912}, he established the same result for two-ended graphs. The following problems, also from \cite{NW91}, remain open:

\begin{prob}[Nash-Williams]\label{p:four}
Is every locally finite graph with exactly one end reconstructible?
\end{prob}

\begin{prob}[Nash-Williams]\label{p:five}
Is every locally finite graph with countably many ends reconstructible?
\end{prob}

In a paper in preparation \cite{BEHLP17}, we will extend the methods developed in the present paper to also construct counterexamples to Problems \ref{p:four} and \ref{p:five}.

This paper is organised as follows. In the next section we will give a short, high-level overview of our counterexample to the Harary-Schwenk-Scott Conjecture. In Section \ref{s:closure}, we will develop the technical tools necessary for our construction, and in Section \ref{s:proof}, we will prove Theorem \ref{t:main}.

For standard graph theoretical concepts we follow the notation in \cite{D10}.

\section{Sketch of the construction}\label{s:outline}

In this section we sketch the main ideas of the construction. For the sake of simplicity we only indicate how to ensure that the trees $T$ and $S$ are vertex-hypomorphic and non-isomorphic, but not that they are edge-hypomorphic as well, nor that neither embeds into the other.

Our plan is to build the trees $T$ and $S$ recursively, where at each step of the construction we ensure for some vertex $v$ already chosen for $T$ that there is a corresponding vertex $w$ of $S$ with $T - v \cong S - w$, or vice versa. This will ensure that by the end of the construction, the trees we have built are hypomorphic. 

More precisely, at step $n$ we will construct subtrees $T_n$ and $S_n$ of our eventual trees, where some of the leaves of these subtrees have been coloured in two colours, say red and blue. 
We will only further extend the trees from these coloured leaves, and we will extend from leaves of the same colour in the same way.

That is, the plan is that there should be two further rooted trees $R$ and $B$ such that $T$ can be obtained from $T_n$ by attaching copies of $R$ at all red leaves and copies of $B$ at all blue leaves, and $S$ can be obtained from $S_n$ in the same way. At step $n$, however, we do not yet know what these trees $R$ and $B$ will eventually be. 

Nevertheless, we can ensure that the induced subgraphs, $T-v$ and $S-w$, of the vertices we have dealt with so far really will match up. More precisely, by step $n$ we have vertices $x_1, \ldots, x_n$ of $T_n$ and $y_1, \ldots, y_n$ of $S_n$ for which we intend that $T - x_j$ should be isomorphic to $S - y_j$ for each $j$. We ensure this by arranging that for each $j$ there is an isomorphism from $T_n - x_j$ to $S_n - y_j$ which preserves the colours of the leaves.

The $T_n$ will be nested, and we will take $T$ to be the union of all of them; similarly the $S_n$ will be nested and we take $S$ to be the union of all of them.

There is a trick to ensure that $T$ and $S$ do not end up being isomorphic. 
First we ensure, for each $n$, that there is no isomorphism from $T_n$ to $S_n$. 
We also ensure that the part of $T$ or $S$ beyond any coloured leaf of $T_n$ or $S_n$ begins with a long non-branching path (called a \emph{bare} path), longer than any such path appearing in $T_n$ or $S_n$. Call the length of these long paths $k_{n+1}$.

Suppose now for a contradiction that there is an isomorphism from $T$ to $S$. 
Then there must exist some large $n$ such that the isomorphism sends some vertex $t$ of $T_n$ to a vertex $s$ of $S_n$. 
However, $T_n$ is the component of $T$ containing $t$ after all bare paths of length $k_{n+1}$ have been removed\footnote{Here and throughout this section we will omit minor technical details for brevity.}, and so it must map isomorphically onto the component of $S$ containing $s$ after all bare paths of length $k_{n+1}$ have been removed, namely onto $S_n$. 
However, there is no isomorphism from $T_n$ onto $S_n$, so we have the desired contradiction.

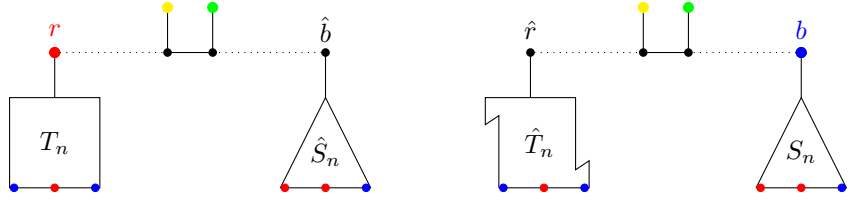
\begin{figure}[ht!]
\begin{subfigure}[t]{0.5\textwidth}
\centering
\begin{tikzpicture}[scale=0.6]

\node[draw, red, circle,scale=.4, fill] (Ttop) at (1,3) {};
\node[text=red] at (1,3.5) {$r$};
\node[] at (7,3.5) {$\hat{b}$};

\node[draw, circle,scale=.3, fill] (Stop) at (7,3) {};
  
\draw (0,0) -- (2,0) -- (2,2) -- (0,2) -- (0,0);
\draw (1,2) -- (1,3);

\node at (1,1) {$T_n$};
  
\draw (6,0) -- (8,0) -- (7,2) -- (6,0);
\draw (7,2) -- (7,3);

\node at (7,0.8) {$\hat{S}_n$};

\node[draw, red, circle,scale=.3, fill] () at (1,.0) {};
\node[draw, blue, circle,scale=.3, fill] () at (0.1,.0) {};
\node[draw, blue, circle,scale=.3, fill] () at (1.9,.0) {};

\node[draw, red, circle,scale=.3, fill] () at (6.1,.0) {};
\node[draw, red, circle,scale=.3, fill] () at  (7,.0) {};
\node[draw, blue, circle,scale=.3, fill] () at (7.9,.0) {};
  
\node[draw, circle,scale=.3, fill] (gadget1) at (3.5,3) {};
\node[draw, circle,scale=.3, fill] (gadget2) at (4.5,3) {};
\draw (gadget1) -- (gadget2);
    
\draw[dotted] (Ttop) -- (gadget1);
\draw[dotted] (Stop) -- (gadget2);
    
\draw(gadget1) -- (3.5,4);
\node[draw, yellow, circle,scale=.4, fill] (leaf1) at (3.5,4) {};
\node[draw, green, circle,scale=.4, fill] (leaf1) at (4.5,4) {};
\draw (gadget2) -- (leaf1);
\node[draw, red, circle,scale=.4, fill] (Ttop) at (1,3) {};     
\end{tikzpicture}

\end{subfigure}%
\begin{subfigure}[t]{0.5\textwidth}
\centering
\begin{tikzpicture}[scale=0.6]

\node[draw, circle,scale=.3, fill] (Ttop) at (1,3) {};
\node[text=blue] at (7,3.5) {$b$};
\node[] at (1,3.5) {$\hat{r}$};

\node[draw, blue, circle,scale=.4, fill] (Stop) at (7,3) {};
  
\draw (0.3,0) -- (2.3,0) -- (2.3,0.6) -- (2,0.4)-- (2,2) -- (0,2) -- (0,1.4)--(0.3,1.6)--(0.3,0);
\draw (1,2) -- (1,3);

\node at (1.2,1) {$\hat{T}_n$};
  
\draw (6,0) -- (8,0) -- (7,2) -- (6,0);
\draw (7,2) -- (7,3);

\node at (7,0.8) {$S_n$};

\node[draw, red, circle,scale=.3, fill] () at (1.3,.0) {};
\node[draw, blue, circle,scale=.3, fill] () at (0.4,.0) {};
\node[draw, blue, circle,scale=.3, fill] () at (2.2,.0) {};

\node[draw, red, circle,scale=.3, fill] () at (6.1,.0) {};
\node[draw, red, circle,scale=.3, fill] () at  (7,.0) {};
\node[draw, blue, circle,scale=.3, fill] () at (7.9,.0) {};
  
\node[draw, circle,scale=.3, fill] (gadget1) at (3.5,3) {};
\node[draw, circle,scale=.3, fill] (gadget2) at (4.5,3) {};
\draw (gadget1) -- (gadget2);
    
\draw[dotted] (Ttop) -- (gadget1);
\draw[dotted] (Stop) -- (gadget2);
    
\draw(gadget1) -- (3.5,4);
\node[draw, yellow, circle,scale=.4, fill] (leaf1) at (3.5,4) {};
\node[draw, green, circle,scale=.4, fill] (leaf1) at (4.5,4) {};
\draw (gadget2) -- (leaf1);
\node[draw, blue, circle,scale=.4, fill] (Stop) at (7,3) {}; 
\end{tikzpicture}
\end{subfigure}
\caption{A first approximation of $T_{n+1}$ on the left, and $S_{n+1}$ on the right. All dotted lines are non-branching paths of length  $k_{n+1}$.}
\label{sketchfig1}
\end{figure}

Suppose now that we have already constructed $T_n$ and $S_n$ and wish to construct $T_{n+1}$ and $S_{n+1}$. Suppose further that we are given a vertex $v$ of $T_n$ for which we wish to find a partner $w$ in $S_{n+1}$ so that $T - v$ and $S - w$ are isomorphic. We begin by building a tree $\hat{T}_n \not \cong T_n$ which has some vertex $w$ such that $T_n - v \cong \hat{T}_n - w$. This can be done by taking the components of $T_n - v$ and arranging them suitably around the new vertex $w$.

We will take $S_{n+1}$ to include $S_n$ and $\hat{T}_n$, with the copies of red and blue leaves in $\hat{T}_n$ also coloured red and blue respectively. As indicated on the right in Figure \ref{sketchfig1}, we add paths of length $k_{n+1}$ to some blue leaf $b$ of $S_n$ and to some red leaf $\hat{r}$ of $\hat{T}_n$ and join these paths at their other endpoints by some edge $e_n$. We also join two new leaves $y$ and $g$ to the endvertices of $e_n$. We colour the leaf $y$ yellow and the leaf $g$ green (to avoid confusion with the red and blue leaves from step $n$, we take the two colours applied to the leaves in step $n+1$ to be yellow and green).

To ensure that $T_{n+1} - v \cong S_{n+1} - w$, we take $T_{n+1}$ to include $T_n$ together with a copy $\hat{S}_n$ of $S_n$, coloured appropriately and joined up in the same way, as indicated on the left in Figure \ref{sketchfig1}.

The only problem up to this point is that we have not been faithful to our intention of extending in the same way at each red or blue leaf of $T_n$ and $S_n$. Thus, we now copy the same subgraph appearing beyond $r$ in Fig.~\ref{sketchfig1}, including its coloured leaves,  onto all the other red leaves of $S_n$ and $T_n$. Similarly we copy the subgraph appearing beyond the blue leaf $b$ of $S_n$ onto all other blue leaves of $S_n$ and $T_n$.

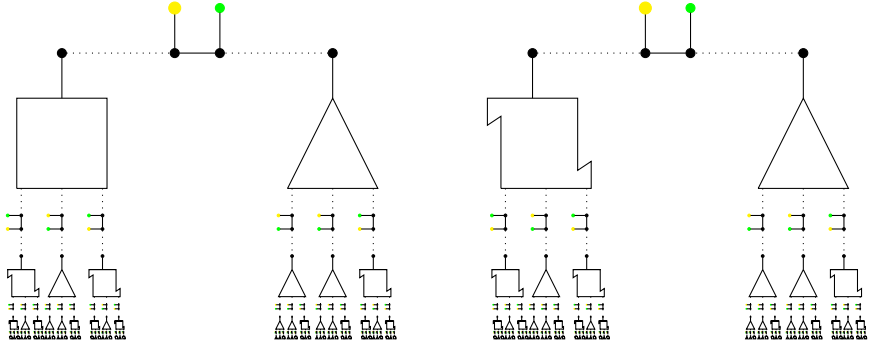
\begin{figure}[ht!]
\begin{subfigure}[t]{0.5\textwidth}
\centering
\begin{tikzpicture}[x=.5cm,y=.5cm, scale=1.2, every node/.style={transform shape}]
\node[draw, circle,scale=.3, fill] (Ttop) at (1,3) {};

\node[draw, circle,scale=.3, fill] (Stop) at (7,3) {};
  
\draw (0,0) -- (2,0) -- (2,2) -- (0,2) -- (0,0);
\draw (1,2) -- (1,3);

\draw (6,0) -- (8,0) -- (7,2) -- (6,0);
\draw (7,2) -- (7,3);

\pic[scale=0.3, every node/.style={transform shape}] at (0.1,0) {decdecsquare};
\pic[scale=0.3, every node/.style={transform shape}] at (1,0) {decdectriangle};
\pic[scale=0.3, every node/.style={transform shape}] at (1.9,0) {decdecsquare};

\pic[scale=0.3, every node/.style={transform shape}] at (6.1,0) {decdectriangle};
\pic[scale=0.3, every node/.style={transform shape}] at (7,0) {decdectriangle};
\pic[scale=0.3, every node/.style={transform shape}] at (7.9,0) {decdecsquare};
  
\node[draw, circle,scale=.3, fill] (gadget1) at (3.5,3) {};
\node[draw, circle,scale=.3, fill] (gadget2) at (4.5,3) {};
\draw (gadget1) -- (gadget2);
    
\draw[dotted] (Ttop) -- (gadget1);
\draw[dotted] (Stop) -- (gadget2);
    
\draw(gadget1) -- (3.5,4);
\node[draw, yellow, circle,scale=.4, fill] (leaf1) at (3.5,4) {};
\node[draw, green, circle,scale=.3, fill] (leaf1) at (4.5,4) {};
\draw (gadget2) -- (leaf1);
\end{tikzpicture}

\end{subfigure}%
\begin{subfigure}[t]{0.5\textwidth}
\centering
\begin{tikzpicture}[x=.5cm,y=.5cm, scale=1.2, every node/.style={transform shape}]

\node[draw, circle,scale=.3, fill] (Ttop) at (1,3) {};
;

\node[draw, circle,scale=.3, fill] (Stop) at (7,3) {};
  
\draw (0.3,0) -- (2.3,0) -- (2.3,0.6) -- (2,0.4)-- (2,2) -- (0,2) -- (0,1.4)--(0.3,1.6)--(0.3,0);
\draw (1,2) -- (1,3);

\draw (6,0) -- (8,0) -- (7,2) -- (6,0);
\draw (7,2) -- (7,3);

\pic[scale=0.3, every node/.style={transform shape}] at (0.4,0) {decdecsquare};
\pic[scale=0.3, every node/.style={transform shape}] at (1.3,0) {decdectriangle};
\pic[scale=0.3, every node/.style={transform shape}] at (2.2,0) {decdecsquare};

\pic[scale=0.3, every node/.style={transform shape}] at (6.1,0) {decdectriangle};
\pic[scale=0.3, every node/.style={transform shape}] at (7,0) {decdectriangle};
\pic[scale=0.3, every node/.style={transform shape}] at (7.9,0) {decdecsquare};

\node[draw, circle,scale=.3, fill] (gadget1) at (3.5,3) {};
\node[draw, circle,scale=.3, fill] (gadget2) at (4.5,3) {};
\draw (gadget1) -- (gadget2);
    
\draw[dotted] (Ttop) -- (gadget1);
\draw[dotted] (Stop) -- (gadget2);
    
\draw(gadget1) -- (3.5,4);
\node[draw, yellow, circle,scale=.4, fill] (leaf1) at (3.5,4) {};
\node[draw, green, circle,scale=.3, fill] (leaf1) at (4.5,4) {};
\draw (gadget2) -- (leaf1);

\end{tikzpicture}

\end{subfigure}
\caption{A sketch of $T_{n+1}$ and $S_{n+1}$ after countably many steps.}
\label{sketchfig2}
\end{figure}

At this point, we would have kept our promise of adding the same thing behind every red and blue leaf of $T_n$ and $S_n$, and hence would have achieved $T_{n+1} - x_j \cong S_{n+1} - y_j$ for all $j \leq n$. However, by gluing the additional copies to blue and red leaves of $T_n$ and $S_n$, we now have ruined the isomorphism between $T_{n+1} - v$ and $S_{n+1} - w$. In order to repair this, we also have to copy the graphs appearing beyond $r$ and $b$ in Fig.~\ref{sketchfig1} respectively onto all red and blue leaves of $\hat{S}_n$ and $\hat{T}_n$. This repairs $T_{n+1} - v \cong S_{n+1} - w$, but again violates our initial promises. In this way, we keep adding, step by step, further copies of the graphs appearing beyond $r$ and $b$ in Fig.~\ref{sketchfig1} respectively onto all red and blue leaves of everything we have constructed so far.

At every step we preserved the colours of leaves in all newly added copies, so we get new red leaves and blue leaves, and we continue the process of copying onto those new leaves as well. After countably many steps we have dealt with all red or blue leaves. We take these new trees to be $S_{n+1}$ and $T_{n+1}$. They are non-isomorphic, since after removing all long bare paths, $T_{n+1}$ contains $T_n$ as a component, whereas $S_{n+1}$ does not. 

Figure \ref{sketchfig2} shows how $T_{n+1}$ and $S_{n+1}$ might appear.  We have now fulfilled our intention of sticking the same thing onto all red leaves and the same thing onto all blue leaves, but we have also ensured that $T_{n+1} - v \cong S_{n+1} - w$, as desired.

\section{Closure with respect to promises}\label{s:closure}

\label{sec:defpromiseclosure}
In this section, we formalise the ideas set forth in the proof sketch of how to extend a graph so that it looks the same beyond certain sets of leaves.

Given a directed edge $\vec{e}=\vec{xy}$ in some forest $G=(V,E)$, we denote by $G(\vec{e})$ the unique component of $G-e$ containing the vertex $y$. We think of $G(\vec{e})$ as a rooted tree with root $y$. As indicated in the previous section, in order to make $T$ and $S$ hypomorphic at the end, we will often have to guarantee $S(\vec{e}) \cong T(\vec{f})$ for certain pairs of edges $\vec{e}$ and $\vec{f}$.

\begin{defn}[Promise structure]
A \emph{promise structure} $\script{P}=\p{G,\vec{P},\script{L}}$ consists of: 
\begin{itemize}
\item a forest $G$,
\item $\vec{P}=\set{\vec{p}_i}:{i \in I}$ a set of directed edges $\vec{P} \subseteq \vec{E}(G)$, and
\item $\script{L}=\set{L_i}:{i \in I}$ a set of pairwise disjoint sets of leaves of $G$.
\end{itemize}
\end{defn}
Often, when the context is clear, we will not make a distinction between $\script{L}$ and the set $\bigcup_i L_i$, for notational convenience. 

We will call an edge $\vec{p}_i \in \vec{P}$ a \emph{promise edge}, and leaves $\ell \in L_i$ \emph{promise leaves}. A promise edge $\vec{p_i} \in \vec{P}$ is called a \emph{placeholder-promise} if the component $G(\vec{p_i})$ consists of a single leaf $c_i \in L_i$, then called a \emph{placeholder-leaf}. We write 
\[
\script{L}_p =\set{ L_i}:{i\in I,\ \vec{p_i} \text{ a placeholder-promise}} \text{    and    } \script{L}_q = \script{L} \setminus \script{L}_p.
\]

Given a leaf $\ell$ in $G$, there is a unique edge $q_{\ell} \in E(G)$ incident with $\ell$, and this edge has a natural orientation $\vec{q_{\ell}}$ towards $\ell$. Informally, we think of the `promise' $\ell \in L_i$ as saying that if we extend $G$ to a graph $H \supset G$, we will do so in such a way that $H(\vec{q_{\ell}}) \cong H(\vec{p_i})$. Given a promise structure $\script{P}=\p{G,\vec{P},\script{L}}$, we would like to construct a graph $H \supset G$ which satisfies all the promises in $\script{P}$. This will be done by the following kind of extension.

\begin{defn}[Leaf extension]
Given an inclusion $H\supseteq G$ of forests and a set $L$ of leaves of $G$, $H$ is called a \emph{leaf extension}, or more specifically an \emph{$L$-extension, of $G$}, if:
\begin{itemize}
\item every component of $H$ contains precisely one component of $G$, and
\item for every vertex $h \in H \setminus G$ and every vertex $g \in G$ in the same component as $h$, the unique $g-h$ path in $H$ meets $L$.
\end{itemize}
\end{defn}

 In the remainder of this section we describe a construction of a forest \emph{$\cl(G)$} which has the following properties.

\begin{prop}\label{p:closure}
Let $G$ be a forest and let $\p{G,\vec{P},\script{L}}$ be a promise structure. Then there is a forest $\cl(G)$ such that:
\begin{enumerate}[series=properties,label={\upshape(cl.\arabic{*})}]
\item\label{i:extension} $\cl(G)$ is an $\script{L}_q$-extension of $G$, and
\item\label{i:keepspromises}  for every $\vec{p}_i \in \vec{P}$ and all $\ell \in L_i$, $$\cl(G)(\vec{p}_i) \cong \cl(G)(\vec{q}_{\ell})$$ are isomorphic as rooted trees.
\end{enumerate}
\end{prop}

We first describe the construction of $\cl(G)$, and then verify the properties asserted in Proposition \ref{p:closure}. Let us define a sequence of promise structures $\p{H^{(i)},\vec{P},\script{L}^{(i)}}$ as follows. We set $\p{H^{(0)},\vec{P},\script{L}^{(0)}} = \p{G,\vec{P},\script{L}}$. We construct a sequence of graphs 
$$G=H^{(0)} \subseteq H^{(1)} \subseteq H^{(2)} \subseteq \cdots,$$ 
and each $H^{(n)}$ will get a promise structure whose set of promise edges is equal to $\vec{P}$ again, yet whose set of promise leaves depends on $n$ as follows: given $\p{H^{(n)},\vec{P},\script{L}^{(n)}}$, we construct $H^{(n+1)}$ by gluing, for each $i$, at every promise leaf $\ell \in L^{(n)}_i$ a rooted copy of $G(\vec{p}_i)$. As promise leaves for $H^{(n+1)}$ we take all promise leaves from the newly added copies of $G(\vec{p}_i)$. That is, if a leaf $\ell \in G(\vec{p_i})$ was such that $\ell \in L_j$, then every copy of that leaf will be in $L^{(n+1)}_j$.

Formally, suppose that $\p{G,\vec{P},\script{L}}$ is a promise structure. For each $\vec{p_i} \in \vec{P}$ let $C_i = G(\vec{p_i})$ and let $c_i$ be the root of this tree. If $U$ is a set and $H$ is a graph, then we denote by $U \times H$ the graph whose vertices are pairs $(u, v)$ with $u \in U$ and $v$ a vertex of $H$, and with an edge from $(u, v)$ to $(u, w)$ whenever $vw$ is an edge of $H$. Let $\p{H^{(0)},\vec{P},\script{L}^{(0)}} = \p{G,\vec{P},\script{L}}$ and given $\p{H^{(n)},\vec{P},\script{L}^{(n)}}$ let us define:
\begin{itemize}
\item $H^{(n+1)}$ to be the quotient of $H^{(n)} \sqcup \bigsqcup_{i \in I} (L^{(n)}_i \times C_i)$ w.r.t.\@ the relation
$$l \sim (l,c_i) \text{ for } l \in L^{(n)}_i \in \script{L}^{(n)}.$$
\item $\script{L}^{(n+1)}=\set{L^{(n+1)}_i}:{i \in I}$ with $L^{(n+1)}_i = \bigcup_{j \in I} L^{(n)}_j \times \p{C_j \cap L_i}$.
\end{itemize}

There is a sequence of natural inclusions $G = H^{(0)} \subseteq H^{(1)} \subseteq \cdots$ and we define \emph{$\cl(G)$} to be the direct limit of this sequence. 

\begin{defn}[Promise-respecting map]
\label{promiserespectingmap}
Let $G$ be a forest, $F^{(1)}$ and $F^{(2)}$ be leaf extensions of $G$, and $\script{P}^{(1)} = \p{F^{(1)},\vec{P},\script{L}^{(1)}}$ and $\script{P}^{(2)} = \p{F^{(2)},\vec{P},\script{L}^{(2)}}$ be promise structures with $\vec{P} \subseteq \vec{E}(G)$. Suppose $X^{(1)} \subseteq V(F^{(1)})$ and $X^{(2)} \subseteq V(F^{(2)})$. 

A bijection $\varphi \colon X^{(1)} \rightarrow X^{(2)}$ is \emph{$\vec{P}$-respecting} (with respect to $\script{P}^{(1)}$ and $\script{P}^{(2)}$) if the image of $L^{(1)}_i \cap X^{(1)}$ under $\varphi$ is $L^{(2)}_i \cap X^{(2)}$ for all $i$.
\end{defn}

Since both promise structures $\script{P}^{(1)}$ and $\script{P}^{(2)}$ refer to the same edge set $\vec{P}$, we can think of them as defining a $|\vec{P}|$-colouring on some sets of leaves. Then a mapping is $\vec{P}$-respecting if it preserves leaf colours.

\begin{lemma}\label{l:propertiesofH}
Let $\p{G,\vec{P},\script{L}}$ be a promise structure and let $G = H^{(0)} \subseteq H^{(1)} \subseteq \cdots$ be as defined above. Then the following statements hold:
\begin{itemize}
\item $H^{(n)}$ is an $\script{L}_q$-extension of $G$ for all $n$,
\item $\Delta(H^{(n+1)}) = \Delta(H^{(n)})$ for all $n$, and
\item For each $\ell \in L_i \in \script{L}$ there exists a sequence of $\vec{P}$-respecting rooted isomorphisms $\varphi_{\ell,n} \colon H^{(n)}(\vec{p_i}) \rightarrow H^{(n+1)}(\vec{q_{\ell}})$ such that $\varphi_{\ell,n+1}$ extends $\varphi_{\ell,n}$ for all $n \in \N$.
\end{itemize}
\end{lemma}
\begin{proof}
The first two statements are clear. We will prove the third by induction on $n$. To construct $H^{(1)}$ from $G$, we glued a rooted copy of $G(\vec{p_i})$ to each $\ell \in L_i$, keeping all copies of promise leaves. Hence, for any given $\ell \in L_i$, the natural isomorphism 
$\varphi_{\ell,0} \colon G(\vec{p}_i) \rightarrow H^{(1)}(\vec{q_{\ell}})$ is $\vec{P}$-respecting as desired. 

Now suppose that $\varphi_{\ell,n}$ exists for all $\ell \in \script{L}$. To form $H^{(n+1)}(\vec{p_i})$, we glued on a copy of $G(\vec{p_i})$ to each $\ell \in L^{(n)}_i \cap H^{(n)}(\vec{p_i})$, and to construct $H^{(n+2)}(\vec{q_{\ell}})$, we glued on a copy of $G(\vec{p_i})$ to each $\ell \in L^{(n+1)}_i \cap H^{(n+1)}(\vec{q_{\ell}})$, in both cases keeping all copies of promise leaves.

Therefore, since $\varphi_{\ell,n}$ was a $\vec{P}$-respecting rooted isomorphism from $H^{(n)}(\vec{p_i})$ to $H^{(n+1)}(\vec{q_{\ell}})$, we can combine the individual isomorphisms between the newly added copies of $G(\vec{p_i})$ with $\varphi_{\ell,n}$ to form $\varphi_{\ell,n+1}$.
\end{proof}

We can now complete the proof of Proposition \ref{p:closure}.

\begin{proof}[Proof of Proposition \ref{p:closure}]
First, we note that $G \subseteq \cl(G)$, and since each $H^{(n)}$ is an $\script{L}_q$-extension of $G$ for all $n$, so is $\cl(G)$. Also, since each $H^{(n)}$ is a forest it follows that $\cl(G)$ is a forest.

Let us show that $\cl(G)$ satisfies property \ref{i:keepspromises}. Since we have the sequence of inclusions $G = H^{(0)} \subseteq H^{(1)} \subseteq \ldots$, it follows that $\cl(G)(\vec{q_{\ell}})$ is the direct limit of the sequence $H^{(0)}(\vec{q_{\ell}}) \subseteq H^{(1)}(\vec{q_{\ell}}) \subseteq \cdots$ and also $\cl(G)(\vec{p_i})$ is the direct limit of the sequence $H^{(0)}(\vec{p_i}) \subseteq H^{(1)}(\vec{p_i}) \subseteq \cdots$. By Lemma \ref{l:propertiesofH} there is a sequence of rooted isomorphisms  $\varphi_{\ell,n} \colon H^{(n)}(\vec{p_i}) \rightarrow H^{(n+1)}(\vec{q_{\ell}})$ such that $\varphi_{\ell,n+1}$ extends $\varphi_{\ell,n}$, so $\varphi_\ell=\bigcup_n \varphi_{\ell,n}$ is the required isomorphism.
\end{proof}

We remark that it is possible to show that $\cl(G)$ is in fact determined, uniquely up to isomorphism, by the properties \ref{i:extension} and \ref{i:keepspromises}. Also we note that since each $H^{(n)}$ has the same maximum degree as $G$, it follows that $\Delta(\cl(G)) = \Delta(G)$.

There is a natural promise structure on $\cl(G)$ given by the placeholder promises in $\vec{P}$ and their corresponding promise leaves. In the construction sketch from Section~\ref{s:outline}, these leaves corresponded to the yellow and green leaves. We now show how to keep track of the placeholder promises when taking the closure of a promise structure. 

Note that if $\vec{p_i}$ is a placeholder promise, then for each $\p{H^{(n)}, \script{P}, \script{L}^{(n)}}$ we have $L_i^{(n)} \supseteq L_{i}^{(n-1)}$. Indeed, for each leaf in $L_{i}^{(n-1)}$ we glue a copy of the component $c_i$ together with the associated promises on the leaves in this component. However, $c_i$ is just a single vertex, with a promise corresponding to $\vec{p_i}$, and hence $L_i^{(n)} \supseteq L_{i}^{(n-1)}$. For every placeholder promise $\vec{p}_i \in \vec{P}$ we define $\cl(L_i)= \bigcup_n L_i^{(n)}$.

\begin{defn}[Closure of a promise structure]
\label{def_closurestructure} The \emph{closure} of the promise structure $\p{G, \script{P}, \script{L}}$ is the promise structure $\cl(\script{P}) = \p{\cl(G), \cl(\vec{P}),\cl(\script{L})}$, where:
\begin{itemize}
\item $\cl(\vec{P})=\set{\vec{p}_i}:{\text{$\vec{p_i} \in \vec{P}$ is a placeholder-promise}}$, and 
\item $\cl(\script{L}) = \{\cl(L_i) \colon \text{$\vec{p_i} \in \vec{P}$ is a placeholder-promise}\}$.
\end{itemize}
\end{defn}

We note that, since each isomorphism $\varphi_{\ell,n}$ from Lemma \ref{l:propertiesofH} was $\vec{P}$-respecting, it is possible to strengthen Proposition \ref{p:closure} in the following way.

\begin{prop}
Let $G$ be a forest and let $\p{G,\vec{P},\script{L}}$ be a promise structure. Then the forest $\cl(G)$ satisfies:
\begin{enumerate}[resume=properties,label={\upshape(cl.\arabic{*})}]
\item\label{i:keepslabelledpromises}  for every $\vec{p_i} \in \vec{P}$ and every $\ell \in L_i$, $$\cl(G)(\vec{p_i}) \cong \cl(G)(\vec{q_{\ell}})$$ are isomorphic as rooted trees, and this isomorphism is $\cl(\vec{P})$-respecting with respect to $\cl(\script{P})$.
\end{enumerate}
\end{prop}
\begin{proof}
Since each isomorphism $\varphi_{\ell,n} \colon H^{(n)}(\vec{p_i}) \rightarrow  H^{(n+1)}(\vec{q_{\ell}}) $ in Proposition~\ref{l:propertiesofH} is $\vec{P}$-respecting, we have 
$$\varphi_{\ell,n}\p{L_i^{(n)} \cap H^{(n)}(\vec{p_i})} = L_i^{(n+1)} \cap H^{(n+1)}(\vec{q_{\ell}}).$$
For each placeholder promise we have that $\cl(L_i) = \bigcup_n L_i^{(n)}$, and so it follows that
$$\cl(L_i) \,\cap \, \cl(G)(\vec{q_{\ell}}) = \bigcup_n \p{L_i^{(n)} \cap H^{(n)}(\vec{q_{\ell}})}$$
and
$$\cl(L_i) \, \cap \, \cl(G)(\vec{p_i}) = \bigcup_n \p{ L_i^{(n)} \cap H^{(n)}(\vec{p_i})}.$$
From this it follows that $\varphi_{\ell} = \bigcup_{n} \varphi_{l,n}$ is a $\cl(\vec{P})$-respecting isomorphism between $\cl(G)(\vec{p_i})$ and $\cl(G)(\vec{q_{\ell}})$ as rooted trees.
\end{proof}

It is precisely this property \ref{i:keepslabelledpromises} of the promise closure that will allow us, in Claim~\ref{petersclaim} below, to maintain partial hypomorphisms during our recursive construction.

\section{The construction}
\label{s:proof}

In this section we construct two hypomorphic locally finite trees neither of which embed into the other, establishing our main theorem announced in the introduction.

\subsection{Preliminary definitions}

\begin{defn}[Bare path]
A path $P=v_0,v_1,\ldots,v_n$ in a graph $G$ is called a \emph{bare path} if $\deg_G(v_i)=2$ for all internal vertices $v_i$ for $0< i < n$. The path $P$ is a \emph{maximal bare path} (or \emph{maximally bare}) if in addition $\deg_G(v_0) \neq 2 \neq\deg_G(v_n)$. An infinite path $P=v_0,v_1,v_2, \ldots$ is \emph{maximally bare} if $\deg_G(v_0)\neq 2$ and $\deg_G(v_i)=2$ for all $i \geq 1$. 
\end{defn}

\begin{lemma}\label{l:miibound}
Let $T$ be a tree and $e \in E(T)$. If every maximal bare path in $T$ has length at most $k \in \mathbb{N}$, then every maximal bare path in $T - e$ has length at most $2k$.
\end{lemma}
\begin{proof}
We first note that every maximal bare path in $T - e$ has finite length, since any infinite bare path in $T_n - e$ would contain a subpath which is an infinite bare path in $T$. If $P = \{x_0,x_1, \ldots, x_n\}$ is a maximal bare path in $T - e$ which is not a subpath of any maximal bare path in $T$, then there is at least one $1 \leq i \leq n-1$ such that $e$ is adjacent to $x_i$, and since $T$ was a tree, $x_i$ is unique. Therefore, both $\{x_0,x_1,\ldots, x_i\}$ and $\{x_i,x_{i+1},\ldots, x_n\}$ are maximal bare paths in $T$. By assumption both $i$ and $n-i$ are at most $k$, and so the length of $P$ is at most $2k$, as claimed.
\end{proof}

\begin{defn}[Bare extension]
Given a forest $G$, a subset $B$ of leaves of $G$, and a component $T$ of $G$, we say that a tree $\hat{T} \supset T$ is a {\em bare extension} of $T$ at $B$ to length $k$ if $\hat{T}$ can be obtained from $T$ by adjoining, at each vertex $l \in B \cap V(T)$, a new path of length $k$ starting at $l$ and a new leaf whose only neighbour is $l$.
\end{defn}

\begin{figure}[ht!]
\begin{subfigure}[t]{0.5\textwidth}
\centering
\begin{tikzpicture}[scale=0.6]

\def \n {6}
\def \radius {1.5}
\def \length {3.5}

\draw (0,0) circle(\radius);
\foreach \s in {1,...,\n}
{
  \node[draw, blue, circle,scale=.4, fill](N\s) at ({360/\n * (\s - 1)}:\radius) {$$};
   \draw[white] (N\s) -- ({360/\n * (\s - 1)}:\length);
}

\node[] at (0,0) {$T$};

\end{tikzpicture}
\caption*{A tree $T$ with designated leaf set $B$.}
\end{subfigure}%
\begin{subfigure}[t]{0.5\textwidth}
\centering
\begin{tikzpicture}[scale=0.6]

\def \n {6}
\def \radius {1.5}
\def \length {3.5}
\def \llength {2.2}

\draw (0,0) circle(\radius);
\foreach \s in {1,...,\n}
{
  \node[draw, blue, circle,scale=.4, fill](N\s) at ({360/\n * (\s - 1)}:\radius) {$$};
  \draw[dotted] (N\s) -- ({360/\n * (\s - 1)}:\length);
  \node[draw, circle,scale=.3, fill]() at ({360/\n * (\s - 1)}:\length) {$$};
  \node[draw, circle,scale=.3, fill](L\s) at ({(360/\n * (\s - 1)) + 15}:\llength) {$$};
  \draw (N\s) -- (L\s);
}
\node[] at (0,0) {$T$};
     
\end{tikzpicture}
\caption*{A bare extension of $T$ at $B$.}
\end{subfigure}
\caption{Building a bare extension of a tree $T$ at $B$ to length $k$. All dotted lines are maximal bare paths of length  $k$.}
\label{miiextension}
\end{figure}

Note that the new leaves attached to each $l \in B$ ensure that the paths of length $k$ are indeed maximal bare paths.

\begin{defn}[$k$-ball]
For $G$ a subgraph of $H$, the $k$-ball $\Ball_H(G, k)$ is the induced subgraph of $H$ on the set of vertices within distance $k$ of some vertex of $G$.
\end{defn}

\begin{defn}[Binary tree] 
For $k \geq 1$, the \emph{binary tree of height $k$} is the unique rooted tree on $2^k -1 = 1+2+\cdots+ 2^{k-1}$ vertices such that the root has degree $2$, there are $2^{k-1}$ leaves, and all other vertices have degree $3$. 
%
%
By a \emph{binary tree} we mean a binary tree of height $k$ for some $k \in \N$.
\end{defn}

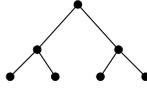
\begin{figure}[ht!]
\centering
\begin{tikzpicture}[scale=0.6]

\node[draw, circle,scale=.3, fill] (blob1) at (0,0) {};
\node[draw, circle,scale=.3, fill] (blob2) at (1,0) {};
\node[draw, circle,scale=.3, fill] (blob3) at (2,0) {};
\node[draw, circle,scale=.3, fill] (blob4) at (3,0) {};

\node[draw, circle,scale=.3, fill] (blob5) at (0.6,.6) {};
\node[draw, circle,scale=.3, fill] (blob6) at (2.4,.6) {};

\draw (blob1) -- (blob5);
\draw (blob2) -- (blob5);
\draw (blob3) -- (blob6);
\draw (blob4) -- (blob6);

\node[draw, circle,scale=.3, fill] (blob7) at (1.5,1.6) {};

\draw (blob5) -- (blob7);
\draw (blob6) -- (blob7);

\end{tikzpicture}
\caption{The binary tree of height $3$.}
\label{binary tree}
\end{figure}

\subsection{The back-and-forth construction} We prove the following theorem.

\begin{theorem}
\label{thm_vertexhypo}
There are two (vertex-)hypomorphic infinite trees $T$ and $S$ with maximum degree $3$ such that there is no embedding $T \hookrightarrow S$ or $S \hookrightarrow T$.
\end{theorem}

To do this we shall recursively construct, for each $n \in \N$, 
\begin{itemize}
\item disjoint (possibly infinite) rooted trees $T_n$ and $S_n$, 
\item disjoint (possibly infinite) sets $R_n$ and $B_n$ of leaves of the forest $T_n \sqcup S_n$,
\item finite sets $X_n \subset V(T_n)$ and $Y_n \subset V(S_n)$, and bijections $\varphi_n \colon X_n \to Y_n$,
\item a family of isomorphisms $\script{H}_n = \set{h_{n,x} \colon T_n - x \to S_n - \varphi_n(x)}:{x \in X_n}$,
\item strictly increasing sequences of integers $k_n \geq 2$ and $b_n \geq 3$,
\end{itemize}
such that (letting all objects indexed by $-1$ be the empty set) for all $n \in \N$:
\begin{enumerate}[label={\upshape($\dagger$\arabic{*})}]
	\item\label{nestedtrees} $T_{n-1} \subset T_{n}$ and $S_{n-1} \subset S_{n}$ as induced subgraphs,
    	\item\label{123} the vertices of $T_n$ and $S_n$ all have degree at most 3, 
        \item\label{colouredroots} the root of $T_n$ is in $R_n$ and the root of $S_n$ is in $B_n$,
        \item \label{nobigbinarytrees} all binary trees appearing as subgraphs of $T_n \sqcup S_n$ are finite and have height at most $b_n$,         

	  	\item\label{kbigenough} all bare paths in $T_n \sqcup S_n$ are finite and have length at most $k_n$,
        
    	\item \label{ballsT}$\Ball_{T_{n}}(T_{n-1}, k_{n-1}+1)$ is a bare extension of $T_{n-1}$ at $R_{n-1} \cup B_{n-1}$ to length $k_{n-1}+1$ and does not meet $R_{n} \cup B_{n}$,

      	\item \label{ballsS}$\Ball_{S_{n}}(S_{n-1}, k_{n-1}+1)$ is a bare extension of $S_{n-1}$ at $R_{n-1} \cup B_{n-1}$ to length $k_{n-1}+1$ and does not meet $R_{n} \cup B_{n}$,
        
	\item \label{noembeddings} there is no embedding from $T_n$ into any bare extension of $S_n$ at $R_n \cup B_n$ to any length, nor from $S_n$ into any bare extension of $T_n$ at $R_n \cup B_n$ to any length,
	\item \label{Tembeddings}any embedding of $T_n$ into a bare extension of $T_n$ at $R_n \cup B_n$ to any length fixes the root of $T_n$ and has image $T_n$,
	\item \label{Sembeddings}any embedding of $S_n$ into a bare extension of $S_n$ at $R_n \cup B_n$ to any length fixes the root of $S_n$ and has image $S_n$,
		\item \label{enumerations}there are enumerations $V(T_n) = \set{t_j}:{j \in J_n}$ and $V(S_n) = \set{s_j}:{j \in J_n}$ such that
		\begin{itemize}
    			\item $J_{n-1} \subset J_{n} \subset \N$, 
                \item $\set{t_j}:{j \in J_{n}}$ extends the enumeration $\set{t_j}:{j \in J_{n-1}}$ of $V(T_{n-1})$, and similarly for $\set{s_j}:{j \in J_n}$,
        			\item $\cardinality{\N \setminus J_n} = \infty$,
    			\item $\Set{0,1,\ldots,n} \subset J_n $,
       		\end{itemize}
 	\item \label{don'tmesswithleaves} $\set{t_j, s_j}:{j \leq n} \cap \p{R_n \cup B_n} = \emptyset$,
	\item\label{XandY} the finite sets of vertices $X_n$ and $Y_n$ satisfy $\cardinality{X_n} = n = \cardinality{Y_n}$, and
		\begin{itemize}
			\item $X_{n-1} \subset X_{n}$ and $Y_{n-1} \subset Y_{n}$,  
			\item $\varphi_{n} \restriction X_{n-1} = \varphi_{n-1}$,
			\item $\set{t_j}:{j \leq n} \subset X_{2n+1}$ and $\set{s_j}:{j \leq n} \subset Y_{2(n+1)}$,
            \item $(X_n \cup Y_n) \cap (R_n \cup B_n) = \emptyset$,
		\end{itemize}
	\item\label{hypomorphism} the families of isomorphisms $\script{H}_n$ satisfy 
		\begin{itemize}
			\item  $h_{n, x} \restriction \p{T_{n-1} - x} = h_{n-1,x}$ for all $x \in X_{n-1}$,   
                         \item the image of $R_n \cap V(T_n)$ under $h_{n,x}$ is $R_n \cap V(S_n)$, and
                         \item the image of $B_n \cap V(T_n)$ under $h_{n,x}$ is $B_n \cap V(S_n)$ for all $x \in X_n$.
		\end{itemize}   
\end{enumerate}

\subsection{
The construction yields the desired non-reconstructible trees.
}
\label{subsec:result}

By property~\ref{nestedtrees}, we have $T_0 \subset T_1 \subset T_2 \subset \cdots $ and $S_0 \subset S_1 \subset S_2 \subset \cdots$. Let $T$ and $S$ be the union of the respective chains. It is clear that $T$ and $S$ are trees, and that as a consequence of \ref{123}, both trees have maximum degree $3$. 

We claim that the map $\varphi =\bigcup_{n} \varphi_n$ is a hypomorphism between $T$ and $S$. Indeed, it follows from \ref{enumerations} and \ref{XandY} that $\varphi$ is a well-defined bijection from $V(T)$ to $V(S)$. To see that $\varphi$ is a hypomorphism, consider any vertex $x$ of $T$. This vertex appears as some $t_j$ in our enumeration of $V(T)$, so by \ref{hypomorphism} the map
$$h_x := \bigcup_{n > 2j } h_{n,x} \colon T-x \to S-\varphi(x)$$ 
is an isomorphism between $T-x$ and $S-\varphi(x)$.

Now suppose for a contradiction that $f \colon T \hookrightarrow S$ is an embedding of $T$ into $S$. Then $f(t_0)$ is mapped into $S_n$ for some $n \in \N$. Properties \ref{kbigenough} and \ref{ballsT} imply that after deleting all maximal bare paths in $T$ of length  $>k_n$, the connected component of $t_0$ is a bare extension of $T_n$ to length 0. Further, by \ref{ballsS}, $\Ball_{S}(S_n, k_n+1)$ is a bare extension of $S_n$ at $R_n \cup B_n$ to length $k_n+1$.
But combining the fact that $f(T_n) \cap S_n \neq \emptyset$ and the fact that $T_n$ does not contain long maximal bare paths, it is easily seen that $f(T_n) \subset \Ball_{S}(S_n, k_n+1)$, contradicting \ref{noembeddings}.\footnote{To get the non-embedding property, we have used \ref{kbigenough}--\ref{noembeddings} at every step $n$. While at the first glance, properties \ref{nobigbinarytrees}, \ref{Tembeddings} and \ref{Sembeddings} do not seem to be needed at this point, they are crucial during the construction to establish \ref{noembeddings} at step $n+1$. See Claim~\ref{floriansclaim} below for details.}

The case $S \hookrightarrow T$ yields a contradiction in a symmetric fashion, completing the proof.

\subsection{
The base case: there are finite rooted trees $T_0$ and $S_0$ satisfying requirements \ref{nestedtrees}--\ref{hypomorphism}.
}
\label{subsec:basecase}

Choose a pair of non-isomorphic, equally sized trees $T_0$ and $S_0$ of maximum degree 3, and pick a leaf each as roots $\rooot{T_0}$ and $\rooot{S_0}$ for $T_0$ and $S_0$, subject to conditions \ref{noembeddings}--\ref{Sembeddings} with $R_0 = \singleton{\rooot{T_0}}$ and $B_0 = \singleton{\rooot{S_0}}$. A possible choice is given in Fig.~\ref{basecasefig}. Here, \ref{noembeddings} is satisfied, because any embedding of $T_0$ into a bare extension of $S_0$ has to map the binary tree of height $3$ in $T_0$ to the binary tree in $S_0$, making it impossible to embed the middle leaf. Properties \ref{Tembeddings} and \ref{Sembeddings} are similar.

\begin{figure}[ht!]
\begin{subfigure}[t]{0.5\textwidth}
\centering
\begin{tikzpicture}[scale=0.6]

\node[draw, circle,scale=.3, fill] (blob1) at (0,0) {};
\node[draw, circle,scale=.3, fill] (blob2) at (1,0) {};
\node[draw, circle,scale=.3, fill] (blob3) at (2,0) {};
\node[draw, circle,scale=.3, fill] (blob4) at (3,0) {};

\node[draw, circle,scale=.3, fill] (blob5) at (0.6,.6) {};
\node[draw, circle,scale=.3, fill] (blob6) at (2.4,.6) {};

\draw (blob1) -- (blob5);
\draw (blob2) -- (blob5);
\draw (blob3) -- (blob6);
\draw (blob4) -- (blob6);

\node[draw, circle,scale=.3, fill] (blob7) at (1.5,1.6) {};

\draw (blob5) -- (blob7);
\draw (blob6) -- (blob7);

\node[draw, circle,scale=.3, fill] (blob8) at (1.5,2.6) {};
\node[draw, circle,scale=.3, fill] (blob9) at (2.5,2.6) {};

\draw (blob7) -- (blob8);
\draw (blob8) -- (blob9);

\node[draw, circle,scale=.3, fill] (blob10) at (1.5,3.6) {};

\node[draw, red, circle,scale=.4, fill] (blob11) at (1.5,4.6) {};
\node[text=red] at (3,4.6) {$\rooot{T_0}$};

\draw (blob8) -- (blob10);
\draw (blob10) -- (blob11);

\end{tikzpicture}
\end{subfigure}%
\begin{subfigure}[t]{0.5\textwidth}
\centering
\begin{tikzpicture}[scale=0.6]

\node[draw, circle,scale=.3, fill] (blob1) at (0,0) {};
\node[draw, circle,scale=.3, fill] (blob2) at (1,0) {};
\node[draw, circle,scale=.3, fill] (blob3) at (2,0) {};
\node[draw, circle,scale=.3, fill] (blob4) at (3,0) {};

\node[draw, circle,scale=.3, fill] (blob5) at (0.6,.6) {};
\node[draw, circle,scale=.3, fill] (blob6) at (2.4,.6) {};

\draw (blob1) -- (blob5);
\draw (blob2) -- (blob5);
\draw (blob3) -- (blob6);
\draw (blob4) -- (blob6);

\node[draw, circle,scale=.3, fill] (blob7) at (1.5,1.6) {};

\draw (blob5) -- (blob7);
\draw (blob6) -- (blob7);

\node[draw, circle,scale=.3, fill] (blob8) at (1.5,2.6) {};

\draw (blob7) -- (blob8);

\node[draw, circle,scale=.3, fill] (blob10) at (1.5,3.6) {};
\node[draw, circle,scale=.3, fill] (blob9) at (2.5,3.6) {};

\draw (blob10) -- (blob9);

\node[draw, blue, circle,scale=.4, fill] (blob11) at (1.5,4.6) {};
\node[text=blue] at (3,4.6) {$\rooot{S_0}$};

\draw (blob8) -- (blob10);
\draw (blob10) -- (blob11);

\end{tikzpicture}
\end{subfigure}
\caption{A possible choice for finite rooted trees $T_0$ and $S_0$.}
\label{basecasefig}
\end{figure}
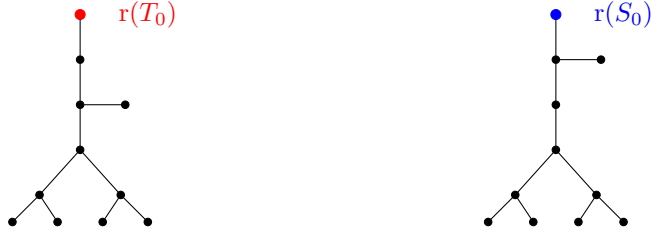

Let $J_0 = \Set{0,1,\ldots, \cardinality{T_0}-1} $ and choose enumerations $V(T_0) = \set{t_j}:{j \in J_0}$ and $V(S_0) = \set{s_j}:{j \in J_0}$ with $t_0 \neq \rooot{T_0}$ and $s_0 \neq \rooot{S_{0}}$. This takes care of \ref{enumerations} and \ref{don'tmesswithleaves}. Finally, \ref{XandY} and \ref{hypomorphism} are satisfied for $X_0 = Y_0 = \script{H}_0 = \varphi_0= \emptyset$. Set $k_0 = 2$ and $b_0 = 3$.

\subsection{The inductive step: set-up}
\label{subsec:setup}

Now, assume that we have constructed trees $T_k$ and $S_k$ for all $k \leq n$ such that \ref{nestedtrees}--\ref{hypomorphism} are satisfied up to $n$. If $n=2m$ is even, then we have $\set{t_j}:{j \leq m-1} \subset X_{n}$, so in order to satisfy \ref{XandY} we have to construct $T_{n+1}$ and $S_{n+1}$ such that the vertex $t_m$ is taken care of in our partial hypomorphism. Similarly, if $n=2m+1$ is odd, then we have $\set{s_j}:{j \leq m-1} \subset Y_{n}$ and we have to construct $T_{n+1}$ and $S_{n+1}$ such that the vertex $s_m$ is taken care of in our partial hypomorphism. Both cases are symmetric, so let us assume in the following that $n=2m$ is even.

Now let $v$ be the vertex with the least index in the set $\set{t_j}:{j \in J_n} \setminus X_n$, i.e.\
\begin{align}
 v = t_i \; \text{ for } \; i = \min \set{\ell}:{t_\ell\in V(T_n) \setminus X_n}. \numberthis\label{defofx}
\end{align}

Then by assumption~\ref{XandY}, $v$ will be $t_m$, unless $t_m$ was already in $X_n$ anyway. In any case, since $\cardinality{X_n}= \cardinality{Y_n} = n$, it follows from \ref{enumerations} that $i \leq n$, so by \ref{don'tmesswithleaves}, $v$ does not lie in our leaf sets $R_n \cup B_n$, i.e.\
\begin{align}
 v \notin R_n \cup B_n. \numberthis\label{xdoesntmess}
\end{align}

In the next sections, we will demonstrate how to to obtain trees $T_{n+1} \supset T_{n}$ and $S_{n+1}\supset S_{n}$ with $X_{n+1}=X_n \cup \singleton{v}$ and $Y_{n+1} = Y_n \cup \singleton{\varphi_{n+1}(v)}$ satisfying \ref{nestedtrees}---\ref{Sembeddings} and \ref{XandY}--\ref{hypomorphism}.

After we have completed this step, since $\cardinality{\N \setminus J_{n}} = \infty$, it is clear that we can extend our enumerations of $T_{n}$ and $S_{n}$ to enumerations of $T_{n+1}$ and $S_{n+1}$ as required, making sure to first list some new elements that do not lie in $R_{n+1} \cup B_{n+1}$.  This takes care of \ref{enumerations} and \ref{don'tmesswithleaves} and completes the recursion step $n \mapsto n+1$.

\subsection{The inductive step: construction}
\label{subsec:constrinducstep}

Given the two trees $T_n$ and $S_n$, we extend each of them through their roots as indicated in Figure~\ref{constructionfigure2} to trees $\tilde{T}_n$ and $\tilde{S}_n$ respectively. The trees $T_{n+1}$ and $S_{n+1}$ will be obtained as components of the promise closure of the forest $G_n = \tilde{T}_n \sqcup \tilde{S}_n$ with respect to the coloured promise edges. 

Since $v$ is not the root of $T_n$, there is a first edge $e$ on the unique path in $T_n$ from $v$ to the root. 
\begin{align}
\text{This edge we also call $e(v)$.}
\numberthis\label{defofedge}
\end{align}
Then $T_n -e$ has two connected components: one that contains the root of $T_n$ which we name $T_n(r)$, and one that contains $v$ which we name $T_n(v)$. 

Since every maximal bare path in $T_n$ has length at most $k_n$ by \ref{kbigenough}, it follows from Lemma~\ref{l:miibound} that all maximal bare paths in $T_n-e$, and so all bare paths in $T_n(r)$ and $T_n(v)$, have bounded length. Let $k=\tilde{k}_n$ be twice the maximum of the length of bare paths in $T_n$, $S_n$, $T_n(r)$ and $T_n(v)$, which exists by \ref{kbigenough}. 

\begin{figure}[ht!]
\begin{subfigure}[t]{0.5\textwidth}
\centering
\begin{tikzpicture}[scale=0.6]

\node[draw, circle,scale=.4, fill] (Ttop) at (1,3) {};
\node at (1.4,3.5) {$\rooot{T_{n}}$};
\node[draw, circle,scale=.3, fill] (extra1) at (0,3) {};
\draw (Ttop) -- (extra1);

\node[draw, circle,scale=.3, fill] (Stop) at (7,3) {};
\node[draw, circle,scale=.3, fill] (extra2) at (8,3) {};
\draw (Stop) -- (extra2);
  
\draw (0,0) -- (2,0) -- (1,2) -- (0,0);
\draw[->,red,ultra thick] (1,2) -- (1,3);

\node at (1.5,-.6) {$T_n$};

\draw (3.33,1) -- (4.33,1) -- (3.83,2) -- (3.33,1);
\node[draw, circle,scale=.3, fill] (blob1) at (3.83,2) {};
\node[draw, circle,scale=.3, fill] (blob2) at (3.83,3) {};
\draw (blob1) -- (blob2);
\node[] at (3.83,.5) {$D_n$};
\node[draw, circle,scale=.3, fill] (blob2) at (4.17,3) {};

\node[draw, circle,scale=.3, fill] (vtx) at (1.05,.4) {};
\node at (1.5,.4) {$v$};

\draw (6,0) -- (8,0) -- (7,2) -- (6,0);
\draw (7,2) -- (7,3);

\node at (8.0,-0.6) {$\hat{S}_n$};

\node[draw, blue, circle,scale=.3, fill] () at (.5,.0) {};
\node[draw, blue, circle,scale=.3, fill] () at (1.2,.0){};
\node[draw, blue, circle,scale=.3, fill] () at (1.7,.0)  {};
\node[draw, red, circle,scale=.3, fill] () at (.7,.0) {};
\node[draw, red, circle,scale=.3, fill] () at (1.4,.0) {};

\node[draw, red, circle,scale=.3, fill] () at (6.5,.0) {};
\node[draw, red, circle,scale=.3, fill] () at (7.2,.0) {};
\node[draw, red, circle,scale=.3, fill] () at (7.7,.0)  {};
\node[draw, blue, circle,scale=.3, fill] () at  (6.7,.0) {};
\node[draw, blue, circle,scale=.3, fill] () at (7.4,.0) {};
  
\node[draw, circle,scale=.3, fill] (gadget1) at (3.5,3) {};
\node[draw, circle,scale=.3, fill] (gadget2) at (4.5,3) {};
\draw (gadget1) -- (gadget2);
    
\draw[dotted] (Ttop) -- (gadget1);
\draw[dotted] (Stop) -- (gadget2);
    
\draw[->,yellow,ultra thick] (gadget1) -- (3.5,4);
\node[draw, yellow, circle,scale=.4, fill] (leaf1) at (3.5,4) {};
\node at (2.7,4.5) {$\rooot{T_{n+1}}$};
\node[draw, green, circle,scale=.3, fill] (leaf1) at (4.5,4) {};
\draw (gadget2) -- (leaf1);
\node at (4.5,4.5) {$g$};
     
\end{tikzpicture}
\caption{tree $\tilde{T}_{n}$}
\end{subfigure}%
\begin{subfigure}[t]{0.5\textwidth}
\centering
\begin{tikzpicture}[scale=0.6]

\node[draw, circle,scale=.3, fill] (Ttop) at (1,3) {};
\node[draw, circle,scale=.3, fill] (extra1) at (0,3) {};
\draw (Ttop) -- (extra1);
\node[draw, circle,scale=.4, fill] (Stop) at (7,3) {};
\node at (6.8,3.5) {$\rooot{S_{n}}$};
\node[draw, circle,scale=.3, fill] (extra2) at (8,3) {};
\draw (Stop) -- (extra2);
  
\draw (0,0) -- (1,0) -- (1.6,1) -- (1,2) -- (0,0);
\draw (1,2) -- (1,3);
\node at (1,-.6) {$\hat{T}_n(r)$};

\node at (7.5,.6) {};

\draw (2,1) -- (3,1) -- (2.5,2) -- (2,1);
\node[draw, circle,scale=.3, fill] (blob3) at (2.5,2) {};
\node[draw, circle,scale=.3, fill] (blob4) at (2.5,3) {};
\draw (blob3) -- (blob4);
\node at (3,2) {$\hat{v}$};
\node[] at (2.6,.5) {$\hat{T}_n(\hat{v})$};
\node[draw, blue, circle,scale=.3, fill] () at (2.2,1){};
\node[draw, red, circle,scale=.3, fill] () at (2.5,1)  {};
\node[draw, blue, circle,scale=.3, fill] () at (2.8,1) {};

\draw (3.33,1) -- (4.33,1) -- (3.83,2) -- (3.33,1);
\node[draw, circle,scale=.3, fill] (blob1) at (3.83,2) {};
\node[draw, circle,scale=.3, fill] (blob2) at (3.83,3) {};
\draw (blob1) -- (blob2);
\node[] at (3.83,.5) {$\hat{D}_n$};
\node[draw, circle,scale=.3, fill] (blob2) at (4.17,3) {};

\draw (6,0) -- (8,0) -- (7,2) -- (6,0);
\draw[->,blue,ultra thick] (7,2) -- (7,3);

\node at (8.0,-.6) {$S_n$};

\node[draw, blue, circle,scale=.3, fill] () at (.5,.0) {};
\node[draw, red, circle,scale=.3, fill] () at (.7,.0) {};

\node[draw, red, circle,scale=.3, fill] () at (6.5,.0) {};
\node[draw, red, circle,scale=.3, fill] () at (7.2,.0) {};
\node[draw, red, circle,scale=.3, fill] () at (7.7,.0)  {};
\node[draw, blue, circle,scale=.3, fill] () at  (6.7,.0) {};
\node[draw, blue, circle,scale=.3, fill] () at (7.4,.0) {};
  
\node[draw, circle,scale=.3, fill] (gadget1) at (3.5,3) {};
\node[draw, circle,scale=.3, fill] (gadget2) at (4.5,3) {};
\draw (gadget1) -- (gadget2);
    
\draw[dotted](Ttop) -- (gadget1);
\draw[dotted] (Stop) -- (gadget2);
    
\draw[->,green,ultra thick] (gadget2) -- (4.5,4);
        
\node[draw, green, circle,scale=.3, fill] (leaf1) at (4.5,4) {};
\node at (5.4,4.5) {$\rooot{S_{n+1}}$};

\node[draw, yellow, circle,scale=.3, fill] (leaf1) at (3.5,4) {};
\node at (3.5,4.5) {$y$};

\draw (gadget1) -- (leaf1);
     
\end{tikzpicture}
\caption*{The tree $\tilde{S}_{n}$.}
\end{subfigure}
\caption{All dotted lines are maximal bare paths of length at least $k=\tilde{k}_n$. The trees $D_n$ are binary trees of height $b_n+3$, hence $D_n\not\hookrightarrow T_n$ and $D_n\not\hookrightarrow S_n$ by (\ref{nobigbinarytrees}).}
\label{constructionfigure2}
\end{figure}
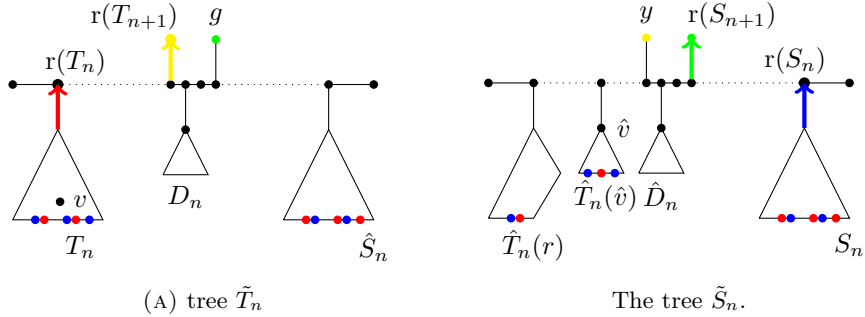

To obtain $\tilde{T}_n$, we extend $T_n$ through its root $\rooot{T_n} \in R_n$ by a path 
$$\rooot{T_n} = u_0, u_1,\dotsc,u_{p-1},u_{p}=\rooot{\hat{S}_n}$$ 
of length $p=4(\tilde{k}_n+1) + 3$, where at its last vertex $u_p$ we glue a rooted copy $\hat{S}_n$ of $S_n$ (via an isomorphism $\hat{w} \leftrightarrow w$), identifying $u_p$ with the root of $\hat{S}_n$.

Next, we add two additional leaves at $u_0$ and $u_p$, so that $\deg(\rooot{T_n})=3=\deg\left(\rooot{\hat{S}_n}\right)$. Further, we add a leaf $\rooot{T_{n+1}}$ at $u_{2k+2}$, which will be our new root for the next tree $T_{n+1}$; and another leaf $g$ at $u_{2k+5}$. Finally, we take a copy $D_n$ of a rooted binary tree of height $b_n+3$ and connect its root via an edge to $u_{2k+3}$. This completes the construction of $\tilde{T}_n$.

The construction of $\tilde{S}_n$ is similar, but with a twist. For its construction, we extend $S_n$ through its root $\rooot{S_n} \in B_n$ by a path
$$\rooot{S_n} = v_p, v_{p-1}, \dotsc,v_1,v_0=\rooot{\hat{T}_n(r)}$$ 
of length $p$, where at its last vertex $v_0$ we glue a copy $\hat{T}_n(r)$ of $T_n(r)$, identifying $v_0$ with the root of $\hat{T}_n(r)$. Then, we take a copy $\hat{T}_n(\hat{v})$ of $T_n(v)$ and connect $\hat{v}$ via an edge to $v_{k+1}$.
\begin{align}
\text{This edge we call $e(\hat{v})$.}
\numberthis\label{defofhatedge}
\end{align}

Finally, as before, we add two leaves at $v_0$ and $v_p$ so that $\deg\left(\rooot{\hat{T}_n(r)}\right)=3=\deg\left(\rooot{S_n}\right)$. Next, we add a leaf $\rooot{S_{n+1}}$ to $v_{2k+5}$, which will be our new root for the next tree $S_{n+1}$; and another leaf $y$ to $v_{2k+2}$. Finally, we take another copy $\hat{D}_n$ of a rooted binary tree of height $b_{n}+3$ and connect its root via an edge to $v_{2k+3}$. This completes the construction of $\tilde{S}_n$.

By the induction hypothesis, certain leaves of $T_n$ have been coloured with one of the two colours $R_n \cup B_n$, and also some leaves of $S_n$ have been coloured with one of the two colours $R_n \cup B_n$. In the above construction, we colour leaves of $\hat{S}_n$, $\hat{T}_n(r)$ and $\hat{T}_n(\hat{v})$ accordingly:
\begin{align}
\begin{split}
\tilde{R}_n &= \p{R_n \cup \set{\hat{w} \in \hat{S}_n \cup \hat{T}_n(r) \cup \hat{T}_n(\hat{v})}:{w \in R_n}} \setminus \Set{\rooot{T_n},\rooot{\hat{T}_n(r)}},  \nonumber \\
\tilde{B}_n &= \p{B_n \cup \set{\hat{w} \in \hat{S}_n \cup \hat{T}_n(r) \cup \hat{T}_n(\hat{v})}:{w \in B_n}} \setminus \Set{\rooot{S_n},\rooot{\hat{S}_n}}.
\end{split}
\numberthis\label{defoftildeBn}
\end{align}

Now put $G_n := \tilde{T}_n \sqcup \tilde{S}_n$ and consider the following promise structure $\script{P}=\p{G_n, \vec{P}, \script{L}}$ on $G_n$, consisting of four promise edges $\vec{P} = \Set{\vec{p}_1,\vec{p}_2,\vec{p}_3, \vec{p}_4}$ and corresponding leaf sets $\script{L}=\Set{L_1, L_2, L_3, L_4}$, as follows:
\begin{align} 
\begin{split}
\bullet \; \vec{p}_1 &\text{ pointing in } T_{n} \text{ towards the root }\rooot{T_n}, \text{ with } L_1=\tilde{R}_n, \\
\bullet \; \vec{p}_2 &\text{ pointing in }S_n \text{ towards the root }\rooot{S_n}, \text{ with } L_2=\tilde{B}_n, \\
\bullet \; \vec{p}_3 &\text{ pointing in }\tilde{T}_n \text{ towards the root }\rooot{T_{n+1}}, \text{ with } L_3=\Set{\rooot{T_{n+1}},y}, \\
\bullet \; \vec{p}_4 &\text{ pointing in } \tilde{S}_n \text{ towards the root } \rooot{S_{n+1}}, \text{ with } L_4=\Set{\rooot{S_{n+1}}, g}.
\end{split}
\numberthis\label{defofpromises}
\end{align}

Note that our construction so far has been tailored to provide us with a $\vec{P}$-respecting isomorphism
\begin{align}
h \colon \tilde{T}_n - v \to \tilde{S}_n - \hat{v}. \numberthis \label{defofhn+1}
\end{align}

Consider the closure $\cl(G_n)$ with respect to the promise structure $\script{P}$ defined above. 
Since $\cl(G_n)$ is a leaf-extension of $G_n$, it has two connected components, just as $G_n$. We now define
\begin{align}
\begin{split}
T_{n+1} &= \text{ the component containing } T_n \text{ in } \cl(G_n), \text{ and} \\
S_{n+1} &= \text{ the component containing } S_n \text{ in } \cl(G_n).
\end{split}
\numberthis\label{defofTnplus1andSnplus1}
\end{align}
It follows that $\cl(G_n) = T_{n+1} \sqcup S_{n+1}$ and $\hat{v} \in V(S_{n+1})$. Further, since $\vec{p}_3$ and $\vec{p}_4$ are placeholder promises, $\cl(G)$ carries a corresponding promise structure, see Definition~\ref{def_closurestructure}. We define 
\begin{align}
R_{n+1} = \cl(L_3) \; \text{ and } \; B_{n+1} = \cl(L_4).
\numberthis\label{defofKnplus1andLnplus1}
\end{align}
Lastly, we set
\begin{align}
\begin{split}
X_{n+1} &= X_n \cup \singleton{v}, \\
Y_{n+1} &= Y_n \cup \singleton{\hat{v}}, \text{ and} \\
\varphi_{n+1} &= \varphi_n \cup \Set{(v,\hat{v})},
\end{split}\numberthis \label{nextXandY}
\end{align}
and put
\begin{align}
k_{n+1} = 2\tilde{k}_n + 3 \; \text{ and } \; b_{n+1} = b_n + 3 \numberthis \label{newknbn}
\end{align}
The construction of trees $T_{n+1}$ and $S_{n+1}$, coloured leaf sets $R_{n+1}$ and $B_{n+1}$, the bijection $\varphi_{n+1} \colon X_{n+1} \to Y_{n+1}$, and integers $k_{n+1}$ and $b_{n+1}$ is now complete. In the following, we verify that  \ref{nestedtrees}--\ref{hypomorphism} are indeed satisfied for the $(n+1)^{\text{th}}$ instance.

\subsection{The inductive step: verification}
\label{subsec:verification} 

\begin{claim}
$T_{n+1}$ and $S_{n+1}$ extend $T_{n}$ and $S_{n}$. Moreover, they are rooted trees of maximum degree 3 such that their respective roots are contained in $R_{n+1}$ and $B_{n+1}$. Hence, \ref{nestedtrees}--\ref{colouredroots} are satisfied. 
\end{claim}
\begin{proof}
Property~\ref{nestedtrees} follows from \ref{i:extension}, i.e.\ that $\cl(G_n)$ is a leaf-extension of $G_n$. Thus, $T_{n+1}$ is a leaf extension of $\tilde{T_n}$, which in turn is a leaf extension of $T_n$, and similar for $S_n$. This shows \ref{nestedtrees}.

As noted after the proof of Proposition \ref{p:closure}, taking the closure does not affect the maximum degree, i.e.\ $\Delta(\cl(G_n)) = \Delta(G_n)=3$. This shows \ref{123}. 

Finally, (\ref{defofKnplus1andLnplus1}) implies \ref{colouredroots}, as $\rooot{T_{n+1}} \in R_{n+1}$ and $\rooot{S_{n+1}} \in B_{n+1}$.
\end{proof}

\begin{claim}\label{binarytrees}
All binary trees appearing as subgraphs of $T_{n+1} \sqcup S_{n+1}$ have height at most $b_{n+1}$, and every such tree of height $b_{n+1}$ is some copy $D_{n}$ or $\hat{D}_{n}$. Hence, $T_{n+1}$ and $S_{n+1}$ satisfy \ref{nobigbinarytrees}.
\end{claim}
\begin{proof}
We first claim that all binary trees appearing as subgraphs of $\tilde{T}_n \sqcup \tilde{S}_n$ which are not contained in $D_n$ or $\hat{D}_n$ have height at most $b_n + 1$. Indeed, note that any binary tree appearing as a subgraph of $T_n$, $\hat{T}_n(r)$, $\hat{T}_n(v)$, $\hat{S}_n$ or $S_n$ has height at most $b_n$ by the inductive hypothesis. Since the paths we added to the roots of $T_n$ and $\hat{S}_n$ to form $\tilde{T}_n$ were sufficiently long, any binary tree appearing as a subgraph of $\tilde{T}_n$ can only meet one of $T_n$, $\hat{S}_n$ or $D_n$. Since the roots of $T_n$ and $\hat{S}_n$ are adjacent to two new vertices in $\tilde{T}_n$, one of degree $1$, any such tree meeting $T_n$ or $\hat{S}_n$ must have height at most $b_n + 1$. By Figure \ref{constructionfigure2} we see that any binary tree in $\tilde{T}_n$ which meets $D_n$ but whose root lies outside of $D_n$ has height at most $3 \leq b_n +1$. Consider then a binary tree whose root lies inside $D_n$, but that is not contained in $D_n$. Again, by Figure \ref{constructionfigure2} we see that the root of $D_n$ must lie in one of the bottom three layers of this binary tree. Hence, if the root of this tree lies on the $k$th level of $D_n$, then the tree can have height at most $\min\{b_n+3 - k, k+2\}$, and hence the tree has height at most $b_n/2 + 2 \leq b_n +1$. Any other binary tree meeting $D_n$ is then contained in $D_n$. It follows that the only binary tree of height $b_n +3$ appearing as a subgraph of $\tilde{T}_n$ is $D_n$, and a similar argument holds for $\tilde{S}_n$ and $\hat{D}_n$.

Recall that $T_{n+1}$ and $S_{n+1}$ are the components of $\cl(\tilde{T}_n \sqcup \tilde{S}_n)$ containing $\tilde{T}_n$ and $\tilde{S}_n$ respectively. If we refer back to Section \ref{s:closure} we see that $T_{n+1}$ can be formed from $\tilde{T_n}$ by repeatedly gluing components isomorphic to $\tilde{T}_n(\vec{p_1})$ or $\tilde{S}_n(\vec{p_2})$ to leaves. Consider a binary tree appearing as a subgraph of $T_{n+1}$ which is contained in $\tilde{T_n}$ or one of the copies of $\tilde{T}_n(\vec{p_1})$ or $\tilde{S}_n(\vec{p_2})$. By the previous paragraph, this tree has height at most $b_n + 3$, and if it has height $b_n + 3$ it is a copy $D_n$ or $\hat{D}_n$. Suppose then that there is a binary tree, of height $b$, whose root is in $\tilde{T_n}$, but is not contained in $\tilde{T}_n$. Such a tree must contain some vertex $\ell \in \tilde{T}_n$ which is adjacent to a vertex not in $\tilde{T}_n$. Hence, $\ell$ must have been a leaf in $\tilde{T}_n$ at which a copy of $\tilde{T}_n(\vec{p_1})$ or $\tilde{S}_n(\vec{p_2})$ was glued on. However, the roots of each of these components are adjacent to just two vertices, one of degree $1$, and hence this leaf $\ell$ must either be in the bottom, or second to bottom layer of the binary tree. Therefore, $b \leq b_n + 2$. A similar argument holds when the root lies in some copy of $\tilde{T}_n(\vec{p_1})$ or $\tilde{S}_n(\vec{p_2})$, and also for $S_{n+1}$.

Therefore, all binary trees appearing as subgraphs of $T_{n+1} \sqcup S_{n+1}$ have height at most $b_n + 3$, and every such tree is some copy $D_{n}$ or $\hat{D}_{n}$. Hence, since $b_{n+1} = b_n + 3$, it follows that $b_{n+1} \geq b_n$ and $T_{n+1}$ and $S_{n+1}$ satisfy \ref{nobigbinarytrees}.
\end{proof}

\begin{claim}
Every maximal bare path in $T_{n+1} \sqcup S_{n+1}$ has length at most $k_{n+1}$. Hence, $T_{n+1}$ and $S_{n+1}$ satisfy \ref{kbigenough}.
\end{claim}
\begin{proof}
We first claim that all maximal bare paths in $\tilde{T}_n \sqcup \tilde{S}_n$ have length at most $2\tilde{k}_n + 3$. Firstly, we note that any maximal bare path which is contained in $T_n$ or $\hat{S}_n$ has length at most $k_n \leq \tilde{k}_n$ by the induction hypothesis. Also, since the roots of $T_n$ and $\hat{S}_n$ have degree $3$ in $\tilde{T}_n$, any maximal bare path is either contained in $T_n$ or $\hat{S}_n$, or does not contain any interior vertices from $T_n$ or $\hat{S}_n$. However, it is clear from the construction that any maximal bare path in $\tilde{T}_n$ that does not contain any interior vertices from $T_n$ or $\hat{S}_n$ has length at most $2\tilde{k}_n + 3$. Similarly, any maximal bare path which is contained in $\hat{T}_n(r)$, $\hat{T}_n(v)$, or $S_n$ has length at most $\tilde{k}_n$ by definition. By the same reasoning as above, any maximal bare path in $\tilde{S}_n$ not contained in $\hat{T}_n(r)$, $\hat{T}_n(v)$, or $S_n$ has length at most $2\tilde{k}_n + 3$.

Again, recall that $T_{n+1}$ can be formed from $\tilde{T_n}$ by repeatedly gluing components isomorphic to $\tilde{T}_n(\vec{p_1})$ or $\tilde{S}_n(\vec{p_2})$ to leaves. Any maximal bare path in $T_{n+1}$ which is contained in $\tilde{T_n}$ or one of the copies of $\tilde{T}_n(\vec{p_1})$ or $\tilde{S}_n(\vec{p_2})$ has length at most $2\tilde{k}_n + 3$ by the previous paragraph. However, since every interior vertex in a maximal bare path has degree two, and the vertices in $T_{n+1}$ at which we, at some point in the construction, stuck on copies of $\tilde{T}_n(\vec{p_1})$ or $\tilde{S}_n(\vec{p_2})$ have degree $3$, any maximal bare path in $T_{n+1}$ must be contained in $\tilde{T_n}$ or one of the copies of $\tilde{T}_n(\vec{p_1})$ or $\tilde{S}_n(\vec{p_2})$. Again, a similar argument holds for $S_{n+1}$. Hence, all maximal bare paths in $T_{n+1} \sqcup S_{n+1}$ have length at most $2\tilde{k}_n + 3$. Therefore, since $k_{n+1} = 2\tilde{k}_n + 3$, it follows that $k_{n+1} \geq k_n$ and $T_{n+1}$ and $S_{n+1}$ satisfy \ref{kbigenough}.
\end{proof}

\begin{claim}
$\Ball_{T_{n+1}}(T_{n}, k_{n}+1)$ is a bare extension of $T_{n}$ at $R_{n} \cup B_{n}$ to length $k_{n}+1$ and does not meet $R_{n+1} \cup B_{n+1}$ and similarly for $S_{n+1}$. Hence, $T_{n+1}$ and $S_{n+1}$ satisfy \ref{ballsT} and \ref{ballsS} respectively.
\end{claim}
\begin{proof}
We will show that $T_{n+1}$ satisfies \ref{ballsT}, the proof that $S_{n+1}$ satisfies \ref{ballsS} is analogous. By Proposition \ref{p:closure}, the tree $T_{n+1}$ is an $\p{ (\tilde{R}_n \cup \tilde{B}_n) \cap V(\tilde{T}_n)}$-extension of $\tilde{T}_n$. Hence $T_{n+1}$ is an 
\begin{align}
\p{ \p{(\tilde{R}_n \cup \tilde{B}_n) \cap V(T_n)} \cup r(T_n)} = \big((R_n \cup B_n) \cap V(T_n)\big)\text{-extension of } T_n. \numberthis \label{bla123}
\end{align}

By looking at the construction of $\cl(G)$ from Section \ref{s:closure}, we see that $T_{n+1}$ is also an $L'$-extension of the supertree $T' \supseteq T_n$ formed by gluing a copy of $\tilde{T}_n(\vec{p_1})$ to every leaf in $R_n \cap V(T_n)$ and a copy of $\tilde{S}_n(\vec{p_2})$ to every leaf in $B_n \cap V(T_n)$, where the leaves in $L'$ are the inherited promise leaves from the copies of $\tilde{T}_n(\vec{p_1})$ and $\tilde{S}_n(\vec{p_2})$. 

However, we note that every promise leaf in $\tilde{T}_n(\vec{p_1})$ and $\tilde{S}_n(\vec{p_2})$ is at distance at least $\tilde{k}_n+1$ from the respective root, and so $\Ball_{T_{n+1}}(T_n, \tilde{k}_n) = \Ball_{T'}(T_n, \tilde{k}_n)$. However, $\Ball_{T'}(T_n,\tilde{k}_n)$ can be seen immediately to be a bare extension of $T_n$ at $R_n \cup B_n$ to length $\tilde{k}_n$, and since $\tilde{k}_n \geq k_n + 1$ it follows that $\Ball_{T_{n+1}}(T_n, k_n + 1)$ is a bare extension of $T_n$ at $R_n \cup B_n$ to length $k_n + 1$ as claimed.

Finally, we note that $R_{n+1} \cup B_{n+1}$ is the set of promise leaves $\cl(\script{L}_n)$. By the same reasoning as before, $\Ball_{T_{n+1}}(T_n, k_n + 1)$ contains no promise leaf in $\cl(\script{L}_n)$, and so does not meet $R_{n+1}  \cup B_{n+1}$ as claimed.
\end{proof}

\begin{claim}
\label{floriansclaim}
Let $U_{n+1}$ be a bare extension of $\cl (G_n) = T_{n+1} \sqcup S_{n+1}$ at $R_{n+1} \cup B_{n+1}$ to any length. Then any embedding of $T_{n+1}$ or $S_{n+1}$ into $U_{n+1}$ fixes the respective root. Hence, $T_{n+1}$ and $S_{n+1}$ satisfy  \ref{noembeddings}.
\end{claim}

\begin{proof}
Recall that the promise closure was constructed by recursively adding copies of rooted trees $C_i$ and identifying their roots with promise leaves. For  the promise structure $\script{P}=\p{G_n, \vec{P}, \script{L}}$ on $G_n$ we have $C_1 = \tilde{T}_n(\vec{p_1})$ and $C_2 = \tilde{S}_n(\vec{p_2})$. 

Note that by \ref{kbigenough}, the image of any embedding $T_n \hookrightarrow U_{n+1}$ cannot contain a bare path of length $k_n + 1$. Also, by construction, every copy of $T_n$, $S_n$, $\hat{T}_n(r)$, or $\hat{T}_n(\hat{v})$ in $T_{n+1}$ has the property that its $(k_n+1)$-ball in $T_{n+1}$ is a bare extension to length $k_n + 1$ of this copy. Hence, if the root of $T_n$ embeds into some copy of $T_n$, $S_n$, $\hat{T}_n(r)$, or $\hat{T}_n(\hat{v})$, then the whole tree embeds into a bare extension of this copy. The same is true for $S_n$.

By \ref{noembeddings}, there are no embeddings of $T_n$ into a bare extension of $S_n$, or of $S_n$ into a bare extension of $T_n$. Moreover, since both $\hat{T}_n(r)$ and $\hat{T}_n(\hat{v})$ are subtrees of $T_n$, there is no embedding of $T_n$ or $S_n$ into bare extensions of them by \ref{noembeddings} and \ref{Tembeddings}. 

Thus, only the following embeddings are possible:
\begin{itemize}
\item $T_n$ embeds into a bare extension of a copy of $T_n$, or $S_n$ embeds into a bare extension of a copy of $S_n$. In both cases, the root must be preserved, as otherwise we contradict \ref{Tembeddings} or \ref{Sembeddings}.
\end{itemize}

Let $f \colon T_{n+1} \hookrightarrow U_{n+1}$ be an embedding. By Claim \ref{binarytrees}, $U_{n+1}$ contains no binary trees of height $b_n + 3$ apart from $D_n$, $\hat{D}_n$, and the copies of those two trees that were created by adding copies of $C_1$ and $C_2$. Consequently $f$ maps $D_n$ to one of these copies, mapping the root to the root. The neighbours of $\rooot{T_{n+1}}$ and $g$ must map to vertices of degree $3$ at distance two and three from the image of the root of $D_n$ respectively, which forces $f(\rooot{T_{n+1}}) \in R_{n+1}$.
If $f(\rooot{T_{n+1}})=\rooot{T_{n+1}}$ then we are done. 

Otherwise there are two possibilities for $f(\rooot{T_{n+1}})$. If $f(\rooot{T_{n+1}})$ is contained in a copy of $C_1$, then $\rooot{T_n}$ maps to a promise leaf other than the root in a copy of $T_n$, $S_n$, $\hat{T}_n(r)$, or $\hat{T}_n(\hat{v})$. If $f(\rooot{T_{n+1}}) = y$ or $f(\rooot{T_{n+1}})$ is contained in a copy of $C_2$, then $\rooot{T_n}$ maps to a copy of $\rooot{\hat{T}_n(r)}$ or some vertex of $\hat{T}_n(\hat{v})$. In both cases the root of $T_n$ does not map to the root of a copy of $T_n$, which is impossible by the first bullet point.

Finally, let $f \colon S_{n+1} \hookrightarrow U_{n+1}$ be an embedding. By the same arguments as above $f(\rooot{S_{n+1}}) \in B_{n+1}$. If $f$ fixes $\rooot{S_{n+1}}$, we are done.

Otherwise we have again two cases. If $f(\rooot{S_{n+1}}) = g$, or $f(\rooot{S_{n+1}})$ is contained in a copy of $C_1$, then $v_{k+1}$ (the neighbour of $\hat{v}$ on the long path) would have to map to a vertex of degree $2$, giving an immediate contradiction.  If $f(\rooot{S_{n+1}})$ is contained in a copy of $C_2$, then $\rooot{S_n}$ maps to a promise leaf other than the root in a copy of $T_n$, $S_n$, $\hat{T}_n(r)$, or $\hat{T}_n(\hat{v})$ which is also impossible by the observations in the bullet points.
\end{proof}

\begin{claim}
Let $U_{n+1}$ be as in Claim~\ref{floriansclaim}. Then there is no embedding of $T_{n+1}$ or $S_{n+1}$ into $U_{n+1}$ whose image contains vertices outside of $\cl (G_n)$, i.e.\ vertices that have been added to form the bare extension. 

Since a root-preserving embedding of a locally finite tree into itself must be an automorphism, this together with the previous claim implies \ref{Tembeddings}and \ref{Sembeddings}.
\end{claim}

\begin{proof}
We prove this claim for $T_{n+1}$, the proof for $S_{n+1}$ is similar. Assume for a contradiction that there is a vertex $w$ of $T_{n+1}$ and an embedding $f \colon T_{n+1} \hookrightarrow U_{n+1}$ such that $ f(w) \notin \cl (G_n)$. 
By definition of bare extension, removing $f(w)$ from $U_{n+1}$ splits the component of $f(w)$ into at most two components, one of which is a path. 

Note first that $w$ does not lie in a copy of $D_n$ or $\hat{D}_n$, because these must map to binary trees of the same height by Claim~\ref{binarytrees}. Furthermore, all vertices in $R_{n+1} \cup B_{n+1}$ have a neighbour of degree $3$ whose neighbours all have degree $\geq 2$, thus $w \notin R_{n+1} \cup B_{n+1}$. Finally, only one component of $T_{n+1}-w$ can contain vertices of degree $3$. Consequently, $w$ must lie in a copy $C$ of $T_n$, $S_n$, $\hat{T}_n(r)$, or $\hat{T}_n(\hat{v})$.  

All maximal bare paths in the image $f(C)$ have length at most $ k=\tilde{k}_n$, so $f(C)$ cannot intersect any copies of $T_n$, $S_n$, $\hat{T}_n(r)$, or $(\hat{T}_n(\hat{v})+v_{k+1})$. Let $r$ be the root of $C$ (where $r=\hat{v}$ in the last case). Now $f(r)$ must have the following properties: it is a vertex of degree $3$, and the root of a nearest binary tree of height $b_{n+1}$ not containing $f(r)$ lies at distance $d$ from $f(r)$, where $5 \leq d \leq 2k+4$.

But the only vertices with these properties are contained in copies of $T_n$, $\hat{S}_n$, $\hat{T}_n(r)$, or $(\hat{T}_n(\hat{v})+v_{k+1})$. This contradicts the fact that $f(C)$ does not intersect any of these copies.
\end{proof}

\begin{claim}
The function $\varphi_{n+1}$ is a well-defined bijection extending $\varphi_n$, such that its domain and range do not intersect $R_{n+1} \cup B_{n+1}$. Hence, property \ref{XandY} holds for $\varphi_{n+1} \colon X_{n+1} \to Y_{n+1}$.
\end{claim}
\begin{proof}
By the choice of $x$ in (\ref{defofx}) and the definition of $\varphi_{n+1}\colon X_{n+1} \to Y_{n+1}$ in (\ref{nextXandY}), the first three items of property~\ref{XandY} hold.

Since $v$ does not lie in $R_n \cup B_n$ by (\ref{xdoesntmess}), it follows by our construction of the promise structure $\script{P}=\p{G_n, \vec{P}, \script{L}}$ in (\ref{defoftildeBn}) and (\ref{defofpromises}) that neither $v$ nor $\hat{v}=\varphi_{n+1}(v)$ appear as promise leaves in $\script{L}$. Furthermore, by the induction hypothesis, $(X_n \cup Y_n) \cap (R_n \cup B_n) = \emptyset$, so no vertex in $(X_n \cup Y_n)$ appears as a promise leaf in $\script{L}$ either. Thus, in formulas, 
\begin{align}
\p{X_{n+1} \cup Y_{n+1}} \cap \bigcup _{L \in \script{L}} L = \emptyset. \numberthis \label{Nocreativityfornamesleft}
\end{align}
In particular, since
$$\p{R_{n+1} \cup B_{n+1}} \cap G_n = \p{\cl(L_3) \cup \cl(L_4) }\cap G_n = L_3 \cup L_4,$$
and $X_{n+1} \cup Y_{n+1} \subset G_n$, we get $(X_{n+1} \cup Y_{n+1}) \cap (R_{n+1}  \cup B_{n+1}) = \emptyset$. Thus, also the last item of \ref{XandY} is verified. 
\end{proof}

\begin{claim}\label{petersclaim}
There is a family of isomorphisms $\script{H}_{n+1} = \set{h_{n+1,x}}:{x \in X_{n+1}}$ witnessing that $T_{n+1} - x$ and $S_{n+1} - \varphi_{n+1}(x)$ are isomorphic for all $x \in X_{n+1}$, such that $h_{n+1,x}$ extends $h_{n,x}$ for all $x \in X_n$.  Hence, property \ref{hypomorphism} holds.
\end{claim}
\begin{proof}
There are four things to be verified for this claim. Firstly, we need an isomorphism $h_{n+1,v}$ witnessing that $T_{n+1} - v$ and $S_{n+1} - \hat{v}$ are isomorphic. Secondly, we need to \emph{extend} all previous isomorphisms $h_{n,x}$ between $T_n - x$ and $S_n - \varphi_n(x)$ to $T_{n+1} - x$ and $S_{n+1} - \varphi_n(x)$. This will take care of the first item of \ref{hypomorphism}. To also comply with the remaining two items, we need to make sure that each isomorphism in 
$$\script{H}_{n+1} = \set{h_{n+1,x}}:{x \in X_{n+1}}$$
maps leaves in $R_{n+1} \cap V(T_{n+1})$ bijectively to leaves in $R_{n+1} \cap V(S_{n+1})$, and similarly for $B_{n+1}$.

To find the first isomorphism, note that by construction of the promise structure $\script{P}=\p{G_n, \vec{P}, \script{L}}$ on $G_n$ in (\ref{defoftildeBn}), and properties~\ref{i:extension} and \ref{i:keepslabelledpromises} of the promise closure, the trees $T_{n+1}$ and $S_{n+1}$ are obtained from $\tilde{T}_n$ and $\tilde{S}_n$ by attaching at every leaf $r \in \tilde{R}_n$ a copy of the rooted tree $\cl(G_n)(\vec{p}_1)$, and by attaching at every leaf $b \in \tilde{B}_n$ a copy of the rooted tree $\cl(G_n)(\vec{p}_2)$.

By (\ref{Nocreativityfornamesleft}), neither $v$ nor $\varphi_{n+1}(v)$ are mentioned in $\script{L}$. As observed in (\ref{defofhn+1}),  there is a $\vec{P}$-respecting isomorphism 
$$h \colon \tilde{T}_n - v \to \tilde{S}_n - \varphi_{n+1}(v).$$
In other words, $h$ maps promise leaves in $L_i \cap V(\tilde{T}_n)$ bijectively to the promise leaves in $L_i \cap V(\tilde{S}_n)$ for all  $i=1,2,3,4$. Our plan is to extend $h$ to an isomorphism between $T_{n+1}-v$ and $S_{n+1}-\varphi_n(v)$ by mapping the corresponding copies of $\cl(G_n)(\vec{p}_1)$ and $\cl(G_n)(\vec{p}_2)$ attached to the various red and blue leaves to each other.

Formally, by \ref{i:keepslabelledpromises} there is for each $\ell \in \p{\tilde{R}_n \cup \tilde{B}_n} \cap V(T)$ a $\cl(\vec{P})$-respecting isomorphism of rooted trees
\[
\cl(G_n)(\vec{q}_{\ell}) \cong \cl(G_n)(\vec{q}_{h(\ell)}).
\]
Therefore, by combining the isomorphism $h$ between $\tilde{T}_n - v$ and $\tilde{S}_n - \varphi_{n+1}(v)$ with these isomorphisms between each $\cl(G_n)(\vec{q}_{\ell})$ and $\cl(G_n)(\vec{q}_{h(\ell)})$ we get a $\cl(\vec{P})$-respecting isomorphism 
$$h_{n+1,v} \colon T_{n+1} - v \to S_{n+1} - \varphi_{n+1}(v).$$
And since $R_{n+1}$ and $B_{n+1}$ have been defined in (\ref{defofKnplus1andLnplus1}) to be the promise leaf sets of $\cl(\script{P})$, by definition of $\cl(\vec{P})$-respecting (Def.~\ref{promiserespectingmap}), the image of $R_{n+1} \cap V(T_{n+1})$ under $h_{n+1,v}$ is $R_{n+1} \cap V(S_{n+1})$, and similarly for $B_{n+1}$.

It remains to extend the old isomorphisms in $\script{H}_n$. As argued in (\ref{bla123}), both trees $T_{n+1}$ and $S_{n+1}$ are leaf extensions of $T_n$ and $S_n$ at $R_n \cup B_n$ respectively. By property \ref{i:keepslabelledpromises}, these leaf extensions are obtained by attaching at every leaf $r \in R_n$ a copy of the rooted tree $\cl(G_n)(\vec{p}_1)$, and similarly by attaching at every leaf $b \in B_n$ a copy of the rooted tree $\cl(G_n)(\vec{p}_2)$.

By induction assumption \ref{hypomorphism}, for each $x \in X_n$ the  isomorphism
$$ h_{n,x} \colon T_n - x \rightarrow S_n - \varphi_n(x) $$
maps the red leaves of $T_n$ bijectively to the red leaves of $S_n$, and the blue leaves of $T_n$ bijectively to the blue leaves of $S_n$. Thus, by property \ref{i:keepslabelledpromises}, there are $\cl(\vec{P})$-respecting isomorphisms of rooted trees
$$ \cl(G_n)(\vec{q}_\ell) \cong \cl(G_n)(\vec{q}_{h_{n,x}(\ell)}) $$
for all $\ell \in (R_n \cup B_n) \cap V(T_n)$. By combining the isomorphism $h_{n,x}$ between $T_n - x$ and $S_n - \varphi_n(x)$ with these isomorphisms between each $\cl(G_n)(\vec{q}_{\ell})$ and $\cl(G_n)(\vec{q}_{h_{n,x}(l)})$, we obtain a $\cl(\vec{P})$-respecting extension 
$$h_{n+1,x} \colon T_{n+1} - x \rightarrow S_{n+1} - \varphi_n(x).$$
As before, by definition of $\cl(\vec{P})$-respecting, the image of $R_{n+1} \cap V(T_{n+1})$ under $h_{n+1,x}$ is $R_{n+1} \cap V(S_{n+1})$, and similarly for $B_{n+1}$.

Finally, by construction we have $h_{n+1,x} \restriction (T_n-x) = h_{n,x}$ for all $x \in X_n$ as desired. The proof is complete.
\end{proof}

\section{The trees are also edge-hypomorphic}

In this final section, we briefly indicate why the trees $T$ and $S$ yielded by our strategy above are automatically edge-hypomorphic: we claim the correspondence 
$$\psi \colon e(x) \mapsto e(\varphi(x))$$
as introduced in (\ref{defofedge}) and (\ref{defofhatedge}) is an edge-hypomorphism between $T$ and $S$. For this, we need to verify that 
\begin{enumerate}[label=(\alph*)]
\item $\psi$ is a bijection between $E(T)$ and $E(S)$, and that 
\item the maps $h_x \cup \Set{\langle x,\varphi(x) \rangle} \colon G - e(x) \to H - e(\varphi(x))$ are isomorphisms.
\end{enumerate}
Regarding (b), observe that the map $h$ as defined in (\ref{defofhn+1}) yields, by construction, also a $\vec{P}$-respecting isomorphism
$$
h \cup \Set{( v,\hat{v} )} \colon \tilde{T}_n - e(v) \to \tilde{S}_n - e(\hat{v}),
$$
and from there, the arguments are entirely the same as in the previous section.

For (a), we use the canonical bijection between the edge set of a rooted tree, and its vertices other than the root; namely the bijection mapping every such vertex to the first edge on its unique path to the root. Thus, given the enumeration of $V(T_n)$ and $V(S_n)$ in \ref{enumerations}, we obtain corresponding enumerations of $E(T_n)$ and $E(S_n)$, and since the rooted trees $T_n$ and $S_n$ are order-preserving subtrees of the rooted trees $T_{n+1}$ and $S_{n+1}$ (cf.\ Figure~\ref{constructionfigure2}), it follows that also our enumerations of $E(T_n)$ and $E(S_n)$ extend the enumerations of $E(T_{n-1})$ and $E(S_{n-1})$ respectively. But now it follows from \ref{XandY} and the definition of $\psi$ that by step $2(n+1)$ we have dealt with the first $n$ edges in our enumerations of $E(T)$ and $E(S)$ respectively.

\bibliographystyle{plain}
\bibliography{reconstruction}
\end{document}